\documentclass[10pt]{article}

\usepackage{amsfonts, epsfig, amsmath, amssymb, color,amsthm, mathabx}
\usepackage{mathrsfs}
\usepackage{mathtools}
\usepackage[english]{babel}

\usepackage{mathrsfs}
\newcommand{\E}{\mathbf{E}}
\renewcommand{\P}{\mathbf{P}}
\renewcommand{\phi}{\varphi}

\usepackage{color}

\theoremstyle{plain}
\newtheorem{thm}{Theorem}[section]
\newtheorem{lem}[thm]{Lemma}
\newtheorem{prp}[thm]{Proposition}
\newtheorem{cor}[thm]{Corollary}

\theoremstyle{definition}
\newtheorem{rem}[thm]{Remark}
\newtheorem{exa}[thm]{Example}

\newcommand{\e}{\varepsilon}

\newcommand{\be}{\begin{equation}}
\newcommand{\ee}{\end{equation}}
\newcommand{\ben}{\begin{equation*}}
\newcommand{\een}{\end{equation*}}

\newcommand{\ba}{\begin{equation}\begin{aligned}}
\newcommand{\ea}{\end{aligned}\end{equation}}

\DeclareMathOperator{\sgn}{sgn}

\newcommand{\ex}{\mathrm{e}}
\newcommand{\di}{\mathrm{d}}

\newcommand{\rF}{\mathscr{F}}

\newcommand{\bI}{\mathbb{I}}

\newcommand{\bR}{\mathbb{R}}

\newcommand{\bZ}{\mathbb{Z}}
\allowdisplaybreaks[4]

\let\oldmarginpar\marginpar
\renewcommand{\marginpar}[1]{\oldmarginpar{\scriptsize\texttt{\color{red}{#1}}}}

\numberwithin{equation}{section}

\begin{document}

\title{Limit theorems for sticky SDEs with local times\\ and applications to stochastic homogenization}

\date{\today}

\author{Olga Aryasova\thanks{Corresponding author}$\,\, ^,$\thanks{Institute of Mathematics,
Friedrich Schiller University Jena, Inselplatz 5, 07743 Jena, Germany \emph{and}
Institute of Geophysics, National Academy of Sciences of Ukraine, Palladin ave.\ 32, 03680 Kyiv-142,
Ukraine \emph{and} Igor Sikorsky Kyiv Polytechnic Institute, Beresteiskyi ave.\ 37, 03056, Kyiv, Ukraine;
\texttt{oaryasova@gmail.com}}\ , \
Ilya Pavlyukevich\thanks{Institute of Mathematics, Friedrich Schiller University Jena, Inselplatz 5,
07743 Jena, Germany; \texttt{ilya.pavlyukevich@uni-jena.de}}\ ,
and 
Andrey Pilipenko\thanks{
University of Geneva, Section de math\'ematiques,
UNI DUFOUR 24, rue du G\'en\'eral Dufour, Case postale 64, 1211 Geneva 4, Switzerland \emph{and}
Institute of Mathematics, National Academy of Sciences of Ukraine, Tereshchenkivska Str.\ 3, 01601, Kyiv, Ukraine;
 \texttt{pilipenko.ay@gmail.com}}}


\maketitle

\begin{abstract}
In this paper, we establish a general convergence theorem for solutions of
multivariate stochastic differential equations with countably many singular terms
expressed as integrals with respect to local times. The processes under consideration describe diffusions
in the presence of semipermeable hyperplane interfaces. These interfaces may become sticky
after applying a random time change that depends on the amount of local time accumulated on each interface.

We show that, as the distance between the interfaces tends to zero, the local-time terms converge
to a limiting homogenized drift term. When the interfaces are sticky, the limiting diffusion also decelerates,
meaning that its diffusion coefficient is effectively reduced.

Such limit theorems illustrate a form of stochastic homogenization for
diffusions evolving in a heterogeneous medium interleaved with semipermeable, sticky interfaces.

\end{abstract}


\smallskip



\section{Setting and the main results}

The classical mathematical theory of homogenization deals with systems whose parameters
fluctuate rapidly on very small spatial scales.
Its original motivation comes from studying composite materials that possess a periodic internal structure.
For instance in heat-conduction problems,
the material's spacial characteristics may vary periodically,
and the limiting behaviour of such systems in the limit as the period tends to zero
can be described by an effective, averaged model.
More precisely, the goal consists in
determining the limit behavior of solutions
$u_\varepsilon$ to boundary value problems such as
$
 -\nabla \big( A_\varepsilon(x/\varepsilon)\,\nabla u_\varepsilon(x) \big) = f
$
in some domain $G \subseteq\bR^d$
with certain boundary conditions on $\partial G$.
The tensor $A_\varepsilon$ is given by  $A_\varepsilon(x)=A(x/\varepsilon)$, where $A(\cdot)$ is supposed to be periodic in $\bR^d$. The objective of homogenization is to identify the effective thermal conductivity $A$
such that the sequence of solutions $u_\varepsilon$ converges to $u$ as $\varepsilon \to 0$,
where $u$ is the solution of the corresponding homogenized equation
$
A \, \Delta u = f.
$

The homogenization problem can be approached using various analytic methods,
including asymptotic expansions or energy methods, see, e.g., \cite{bensoussan1978asymptotic, pavliotis2008multiscale}.
Alternatively, one can exploit the well-known connection between diffusions and second-order
differential operators, employing stochastic approaches. In general,  the literature on
homogenization of regular diffusions is extensive and encompasses both analytic and probabilistic methods;
notable contributions include
\cite{BerJikPap99,chechkin2007homogenization,hairer2008homogenization,JikKozOle94,Makhno-12,pardoux1999homogenization}.

In the present paper, we study a different type of stochastic homogenization problem: the homogenization of a
diffusion process in the presence of narrowly spaced semipermeable and sticky hyperplane interfaces (membranes).

A fundamental example of a diffusion interacting with a membrane is the \emph{skew Brownian motion}, which solves the stochastic
differential equation (SDE)
$X_t = x + W_t + \beta L_t^{0}(X)$, where $L^0$ is the symmetric semimartingale local time of $X$ at zero, and
$\beta\in[-1,1]$ is
the
\emph{permeability} parameter, see \cite{HShepp-81,Lejay-06}. This process
behaves like a standard Brownian motion away from zero, and, roughly speaking, chooses the positive direction
with probability $\frac{1+\beta}{2}$ and the negative direction
with probability $\frac{1-\beta}{2}$.

In our work, we will consider a system of SDEs which contains countably many terms involving the
local times accumulated by the diffusion on a family of interface hyperplanes.

We first set up the mathematical framework within which the problem will be formulated.

Let $\e\in(0,1]$ and $\delta\in(0,1]$ be small parameters,
$\{a_k^\e\}_{k\in\mathbb Z}$ be an increasing sequence of points in $\bR$.
We consider a process $(X^{\e,\delta},Y^{\e,\delta})
=(X^{\e,\delta},Y^{1,\e,\delta},\dots,Y^{d,\e,\delta})$, $X^{\e,\delta}\in \bR$, $Y^{\e,\delta}\in\bR^d$, $d\in\mathbb N_0$,
which is a solution of the
stochastic differential equation 
\ba
\label{e:XY}
X_t^{\e,\delta}&= x+\sum_{l=1}^m \int_0^t \sigma_{l}^0(s, X^{\e,\delta}_s,Y^{\e,\delta}_s)\,\di W^l_s + \int_0^t b^0(s,X^{\e,\delta}_s,Y_s^{\e,\delta})\,\di s
+\sum_{k=-\infty}^\infty \delta \int_0^t \beta(s,a_k^\e,Y_s^{\e,\delta})\,\di L^{a_k^\e}_s (X^{\e,\delta}),\\
Y^{i,\e,\delta}_t&= y^i+\sum_{l=1}^m \int_0^t \sigma^{i}_l(s,X_s^{\e,\delta},Y_s^{\e,\delta})\,\di W^l_s +\int_0^t b^i(s,X_s^{\e,\delta},Y_s^{\e,\delta})\,\di s
+ \sum_{k=-\infty}^\infty \delta\int_0^t  \theta^i(s,a_k^\e,Y_s^{\e,\delta})\,\di L^{a_k^\e}_s (X^{\e,\delta}),\\
& \hspace{10cm} i=1,\dots,d,\ t\in[0,\infty),
\ea
where $W=(W^1,\dots,W^m)$ is a standard $m$-dimensional Brownian motion, $m\in\mathbb N$, and
$L^a(X^{\e,\delta})$ is the  symmetric semimartingale local time of $X^{\e,\delta}$
at $a\in\bR$ (see the end of this Section for precise definitions).

The phase space of the process $(X^{\e,\delta},Y^{\e,\delta})$ may be viewed as $\bR\times\bR^d$,
$d\in\mathbb N_0$, partitioned into thin layers by the countable family of hyperplanes $\{a_k^\e\}\times\bR^d$, $k\in\bZ$.
In the case $d=0$, the component $Y^{\e,\delta}$ is absent, and we simply consider a one-dimensional process $X^{\e,\delta}$
on the real line.

Between the hyperplanes, the process $(X^{\e,\delta},Y^{\e,\delta})$ evolves as a standard diffusion with drift vector
$(b^j)_{0\leq j\leq n}$
and diffusion matrix $(\sigma_l^i)_{0\leq i\leq n, 1\leq l\leq m}$.
Upon hitting a hyperplane $\{a_k^\e\}\times\bR^d$ at a point $(a_k^\e,y)$, it ``chooses'' the positive
or negative $x$-direction with probabilities
$\frac{1+\delta\beta}{2}$ or $\frac{1-\delta\beta}{2}$, respectively, while sliding along the hyperplane
in the direction  $\theta$.
Thus, the function $\theta$ represents the \emph{sliding coefficient} along the interface hyperplane.

To slow down the diffusion upon hitting the interface,
we introduce the continuous additive
functional
\ba
A_t=A_t^{\e,\delta,\lambda}
:=t + \lambda \sum_{k=-\infty}^\infty \int_0^t \gamma(s,a_k^\e,Y_s^{\e,\delta})\,\di L^{a_k^\e}_s (X^{\e,\delta}),
\ea
where
$\gamma=\gamma(t,x,y)$ is a non-negative function and  $\lambda\in(0,1]$
is a small parameter.
Since the mapping $t\mapsto A_t$ is continuous, strictly increasing, $A_0=0$ and diverges to $+\infty$ as $t\to\infty$, its inverse
$A^{-1}$ is well defined and it provides a random time change.

We then define the time-changed process
\ba
\label{e:XYA}
(X^{\e,\delta,\lambda},Y^{\e,\delta,\lambda}):=
(X^{\e,\delta},Y^{\e,\delta})\circ (A^{\e,\delta,\lambda})^{-1}.
\ea
This is a diffusion that slows down upon hitting the membranes
$\{a_k^\e\}\times\bR^d$; in other words, it is a diffusion with semipermeable sticky interfaces.

The aim of this paper is to investigate the asymptotic behavior of
\((X^{\varepsilon,\delta,\lambda}, Y^{\varepsilon,\delta,\lambda})\) as
\(\varepsilon, \delta, \lambda \to 0\).
Depending on the relative rates at which these parameters vanish, we obtain several distinct limiting processes.
The most intriguing regime occurs when
$\varepsilon$, $\delta$ and $\lambda$ are of the
same order.

The model \eqref{e:XY}, \eqref{e:XYA} is primarily motivated by the study of heat propagation in layered composite materials,
where thin sheets with varying permeability are densely interleaved.
For example, \cite{yuksel2012effective} experimentally studied a glass-foam medium reinforced
with aluminium foil and observed that the addition of the foil reduces the effective thermal conductivity at low temperatures.

Similar multilayer structures also appear in chemistry, biology, and physics
(see, e.g., \cite{dudko2004diffusion, grebenkov2010pulsed, moutal2019diffusion, slkezak2021diffusion, Tanner78}).
In particular, molecular diffusion in biological tissues and other micro-heterogeneous media is strongly
affected by the presence of permeable barriers. Such barriers tend to slow down the effective
diffusion over long times relative to free diffusion, even though their influence may
be negligible on short timescales (see, e.g., \cite{dudko2004diffusion} and references therein).

We make the following assumptions about the coefficients in the system \eqref{e:XY} and the location of the membranes.

\smallskip
\noindent

\smallskip
\noindent
\textbf{A}$_{b,\sigma,\beta,\theta}^{C_b}$:
the coefficients $b$, $\sigma$, $\beta$, $\theta$ are continuous and bounded, more precisely,
\ba
b&\in C_b(\bR_+\times\bR\times\bR^d,\bR^{d+1}),\\
\sigma&\in C_b(\bR_+\times\bR\times\bR^d,\bR^{(d+1)\times m}),\\
\beta&\in C_b(\bR_+\times\bR\times\bR^d,\bR),\\
\theta&\in C_b(\bR_+\times\bR\times\bR^d,\bR^d).
\ea
%
%

\smallskip
\noindent
\textbf{A}$^\text{sep}_a$: The interfaces points are well separated, i.e., there is a constant $C\in (0,\infty)$ such that
for all $\e\in(0,1]$
\ba
C^{-1}\leq \inf_k \frac{a_{k+1}^\e- a_{k}^\e }{\e}\leq \sup_k \frac{a_{k+1}^\e- a_k^\e }{\e}\leq C.
\ea

\smallskip
\noindent
\textbf{A}$_{\Sigma^{00}}$:
There is $C\in (0,\infty)$ such that for all $t,x,y$
\ba
\Sigma^{00}(t,x,y)=\sum_{l=1}^m \sigma^0_l(t,x,y)^2\geq C.
\ea

\smallskip
\noindent
\textbf{A}$_{\Sigma}^\text{elliptic}$:
The matrix $\Sigma=(\Sigma^{ij})_{0\leq i,j\leq d}$ with the entries
\ba
\Sigma^{ij}(t,x,y)=\sum_{l=1}^m \sigma^i_l(t,x,y)\sigma^j_l(t,x,y)
\ea
is uniformly positive definite for all $t,x,y$.

As the first main result of the paper, we establish the existence of a weak solution
to equation \eqref{e:XY}, and, in the case of time-independent coefficients that are twice
continuously differentiable, we further prove existence and uniqueness of solutions.

\begin{thm}
\label{t:existU}
1. Let Assumptions $\boldsymbol{\mathrm A}_{b,\sigma,\beta,\theta}^{C_b}$, $\boldsymbol{\mathrm A}^\mathrm{sep}_a$,
$\boldsymbol{\mathrm A}_{\Sigma^{00}}$ hold.
Then, for any $\delta\in(0,1/\|\beta\|_\infty)$ and any $\e\in(0,1]$
the system \eqref{e:XY} has a weak solution $(X^{\e,\delta},Y^{\e,\delta})$.

2. If, in addition, Assumption $\boldsymbol{\mathrm A}_{\Sigma}^\mathrm{elliptic}$ holds,
all coefficients $b$, $\sigma$, $\beta$, $\theta$ are time-independent,
and the functions $\beta$ and $\theta$ are twice continuously differentiable in $(x,y)$
then the solution is unique and enjoys the strong Markov property.
\end{thm}
Since the system \eqref{e:XY} is not a standard SDE, the results of Theorem \eqref{t:existU} are novel and
require significant technical effort.
The proof of existence follows from the compactness argument, and for the uniqueness we locally construct a certain
piece-wise nonlinear
transformation of $(X^{\e,\delta},Y^{\e,\delta})$,
after which we apply the uniqueness result of \cite{gao1993martingale}.

Our next result concerns the limit of $(X^{\e,\delta},Y^{\e,\delta},A^{\e,\delta,\lambda})$ as $\e,\delta,\lambda\to 0$
in a
commensurable scaling regime.

To that end, we impose several additional assumptions on the coefficients.

\smallskip
\noindent
\textbf{A}$_{b,\sigma,\beta,\theta}^{\mathrm{Lip}_b}$:
All coefficients $b,\sigma,\beta,\theta$ are bounded and globally Lipschitz continuous in all variables.

\smallskip
\noindent
\textbf{A}$_\gamma^{C_b}$:
The stickiness parameter $\gamma$ is a bounded continuous non-negative function:
\ba
\gamma&\in C_b(\bR_+\times\bR\times\bR^d,\bR_+).
\ea

\noindent
\textbf{A}$_d^{\mathrm{Lip}_b}$: There is a bounded globally Lipschitz continuous function $d\colon \bR\to \bR_+$,
such that $\inf_{x\in\bR} d(x)>0$ and
the points $\{a_k^\e\}$ satisfy
\ba
\label{e:ad}
\lim_{\e\downarrow 0}\sup_{k\in\bZ} \Big|\frac{a_{k+1}^\e- a_k^\e }{\e} - d(a_k^\e)\Big|=0,
\ea

\begin{rem}
The function $d$ characterizes the ``density'' of interfaces on the $x$-axis separated by a characteristic distance of order $\e$.
For example, the points $\{a_k^\e\}$ defined by the relation
\ba
\label{e:d}
a^\e_k=\int_0^{\e k} d(x)\,\di x,\quad k\in\mathbb Z,
\ea
satisfy \eqref{e:ad}.
\end{rem}

\smallskip
\noindent
\textbf{A}$_{\mathfrak{p}}$: There exists a limit
\ba
\lim_{\delta,\e\downarrow 0}\frac{\delta}{\e}&=\mathfrak{p}\in[0,\infty].
\ea

\smallskip
\noindent
\textbf{A}$_{\mathfrak{q}}$: There exists a limit
\ba
\lim_{\lambda,\e\downarrow 0}\frac{\lambda}{\e}&=\mathfrak{q}\in[0,\infty].
\ea

\begin{thm}
\label{t:A}
Let Assumptions
$\boldsymbol{\mathrm A}_{b,\sigma,\beta,\theta}^{\mathrm{Lip}_b}$,
$\boldsymbol{\mathrm A}_d^{\mathrm{Lip}_b}$,
$\boldsymbol{\mathrm A}_{\Sigma^{00}}$,
be satisfied,
and for each $\e,\delta\in(0,1]$
let $(X^{\e,\delta},Y^{\e,\delta})$
be a (non necessarily unique) weak solution of \eqref{e:XY}.

If $\boldsymbol{\mathrm A}_{\mathfrak{p}}$ holds with $\mathfrak{p}\in [0,\infty)$, then there is weak convergence
\ba
\label{eq:conv_XY}
(X^{\e,\delta},Y^{\e,\delta} ) \Rightarrow (X,Y ),\quad \e,\delta \to 0,
\ea
in $C(\bR_+,\bR\times \bR^d)$, where $(X,Y)$ is a (unique strong) solution of the SDE
\ba
\label{e:XYlim}
X_t&= x+\sum_{l=1}^m \int_0^t \sigma_{l}^0(s,X_s,Y_s)\,\di W^l_s
+ \int_0^t \Big(b^0(s,X_s,Y_s)+\mathfrak{p}\beta(s,X_s,Y_s) \frac{\Sigma^{00}(s,X_s,Y_s)}{d(X_s)}\Big)\,\di s ,\\
Y^{i}_t&= y^i+\sum_{l=1}^m \int_0^t \sigma^{i}_l(s,X_s,Y_s)\,\di W^l_s
+\int_0^t \Big(b^i(s,X_s,Y_s)   +\mathfrak{p}\theta^i(s,X_s,Y_s) \frac{ \Sigma^{00}(s,X_s,Y_s)}{d(X_s)}\Big)\,\di s,
\quad i=1,\dots,d.
\ea
If additionally $\boldsymbol{\mathrm A}_\gamma^{C_b}$ and $\boldsymbol{\mathrm A}_{\mathfrak{q}}$
hold with $\mathfrak{q}\in [0,\infty)$,
then we have convergence of triples
\ba\label{eq:conv_XYA}
(X^{\e,\delta},Y^{\e,\delta},A^{\e,\delta,\lambda}) \Rightarrow (X,Y,A),\quad \e,\delta,\lambda\to 0,
\ea
in $C(\bR_+,\bR\times \bR^d\times\bR_+)$, where $(X,Y)$ satisfies \eqref{e:XYlim} and
\ba\label{e:Alim}
A_t
= \int_0^t \Big(1+\mathfrak{q}\gamma(s,X_s,Y_s)\frac{\Sigma^{00}(s,X_s, Y_s)}{d(X_s)}\Big) \,\di s.
\ea
If $\mathfrak{q} =\infty$ and $\gamma$ is strictly positive, then
  the sequence of inverse functions
  $((A^{\e,\delta, \lambda })^{-1})$ converges in distribution to $0$ as $\e,\delta,\lambda\to 0$.
\end{thm}

\begin{cor}
\label{t:hom}
Let the processes $X,Y, A$ be from Theorem \ref{t:A}.
If $\mathfrak{p}\in [0,\infty)$ and $\mathfrak{q}\in [0,\infty)$ then we have convergence
\ba
\label{eq:limit_hom1}
(X^{\e,\delta},Y^{\e,\delta})\circ (A^{\e,\delta,\lambda})^{-1} \Rightarrow (X,Y)\circ A^{-1}, \quad \e,\delta,\lambda\to 0.
\ea
In particular, the limit process $(\widehat X,\widehat Y):= (X,Y)\circ A^{-1} $
is a (unique weak) solution to the SDE
\ba
\label{e:XYhatlim}
\widehat X_t&= x+\sum_{l=1}^m \int_0^t \frac{\sigma_{l}^0(s,\widehat X_s,\widehat Y_s)}{\sqrt{1+\mathfrak{q}\gamma(s,\widehat X_s,\widehat Y_s)\frac{\Sigma^{00}(s,\widehat X_s,\widehat Y_s)}{d(\widehat X_s)}}}\,\di W^l_s
+ \int_0^t \frac{b^0(s,\widehat X_s,\widehat Y_s)+\mathfrak{p}\beta(s,\widehat X_s,\widehat Y_s) \frac{\Sigma^{00}(s,\widehat X_s,\widehat Y_s)}{d(\widehat X_s)} }{1+\mathfrak{q}\gamma(s,\widehat X_s,\widehat Y_s)\frac{\Sigma^{00}(s,\widehat X_s, \widehat Y_s)}{d(X_s)}} \,\di s ,\\
\widehat Y^{i}_t&= y^i+\sum_{l=1}^m \int_0^t \frac{\sigma_{l}^i(s,\widehat X_s,\widehat Y_s)}{\sqrt{1+\mathfrak{q}\gamma(s,\widehat X_s,\widehat Y_s)\frac{\Sigma^{00}(s,\widehat X_s, \widehat Y_s)}{d(\widehat X_s)}}}\,\di W^l_s
+\int_0^t \frac{b^i(s,\widehat X_s,\widehat Y_s)   +\mathfrak{p}\theta^i(s,\widehat X_s,\widehat Y_s) \frac{ \Sigma^{00}(s,\widehat X_s,\widehat Y_s)}{d(\widehat X_s)}}{1+\mathfrak{q}\gamma(s,\widehat X_s,\widehat Y_s)\frac{\Sigma^{00}(s,\widehat X_s, \widehat Y_s)}{d(\widehat X_s)}}\,\di s,\\
&\quad i=1,\dots,d.
\ea
If $\mathfrak{q} =\infty$ and $\gamma$ is strictly positive,   then
\ba
\label{eq:limit_hom2}
(X^{\e,\delta},Y^{\e,\delta})\circ (A^{\e,\delta,\lambda})^{-1} \Rightarrow
(x,y), \quad \e,\delta,\lambda\to 0.
\ea
\end{cor}
The proof of these results will be given in in Section \ref{sec:proofs}.

The next result treats the degenerate case when
the
spacing between the membranes dominates the microscopic
scale, i.e., when
$\delta/\varepsilon\to \infty$.
In this case, the stochastic fluctuations
average out completely, and the limit becomes deterministic.

\smallskip
\noindent
$\boldsymbol{\mathrm A}_{\mathfrak{r}}$: There exist a limit
\ba
&\lim_{\lambda,\delta\downarrow 0}\frac{\lambda}{\delta}=\mathfrak{r}\in[0,\infty].
\ea

\begin{thm}\label{thm:limitODE}
Let Assumptions
$\boldsymbol{\mathrm A}_{b,\sigma,\beta,\theta}^{\mathrm{Lip}_b}$,
$\boldsymbol{\mathrm A}_d^{\mathrm{Lip}_b}$,
$\boldsymbol{\mathrm A}_{\Sigma^{00}}$, and
$\boldsymbol{\mathrm A}_{\mathfrak{p}}$ with $\mathfrak{p}=\infty$
be satisfied.
Then, we have convergence in distribution
\ba
\label{eq:limit_isODE}
(\widetilde X^{\e,\delta}_t,\widetilde Y^{\e,\delta}_t)_{t\in[0,\infty)}
:=(X^{\e,\delta}_{t \e/\delta},Y^{\e,\delta}_{t \e/\delta})_{t\in[0,\infty)}
\Rightarrow (\widetilde X_t,\widetilde Y_t)_{t\in[0,\infty)}, \quad \e,\delta\to 0.
\ea
where $\widetilde X, \widetilde Y$ satisfy the ODE
\ba
\label{e:XY_1membrane-degenODE}
\widetilde X_t&= x+ \int_0^t \beta(0,\widetilde X_s,\widetilde Y_s) \frac{\Sigma^{00}(0,\widetilde X_s,\widetilde Y_s)}{d(\widetilde X_s)} \,\di s ,\\
\widetilde Y^{i}_t&= y^i+ \int_0^t  \theta^i(0,\widetilde X_s,\widetilde Y_s) \frac{ \Sigma^{00}(0,\widetilde X_s,\widetilde Y_s)}{d(\widetilde X_s)}\,\di s,
&\quad i=1,\dots,d.
\ea
Set 
\ba
\widetilde A^{\e,\delta, \lambda }_t
:=A^{\e,\delta, \lambda }_{t\e/\delta}
=t\e/\delta +\lambda \sum_{k=-\infty}^\infty \int_0^{t\e/\delta} \gamma( s ,a_k^\e,Y_s^{\e,\delta})\,\di L^{a_k^\e}_s (X^{\e,\delta}).
\ea
and assume additionally that $\boldsymbol{\mathrm A}_\gamma^{C_b}$ and $\boldsymbol{\mathrm A}_{\mathfrak{r}}$ hold.

If $\mathfrak{r}\in[0,\infty)$,  then
\ba
\label{eq:conv_XYA_deg}
(\widetilde  X^{\e,\delta},\widetilde  Y^{\e,\delta},\widetilde  A^{\e,\delta, \lambda })
\Rightarrow
(\widetilde  X,\widetilde  Y,\widetilde  A), \quad \e,\delta,\lambda\to 0.
\ea
where
\ba
\widetilde  A_t:=
\int_0^t \mathfrak{r}\gamma(0,\widetilde  X_s, \widetilde  Y_s)
\frac{\Sigma^{00}(0,\widetilde  X_s,\widetilde  Y_s)}{d(\widetilde X_s)}  \,\di s.
 \ea
If $\mathfrak{r} =\infty$ and $\gamma$ is strictly positive,
then the sequence of inverse functions
  $((\widetilde A^{\e,\delta, \lambda })^{-1})$ converges in distribution to $0$ as $\e,\delta,\lambda\to 0$.
 \end{thm}

\begin{rem}
To carry out the random time change of the processes on the time scale $t\e/\delta$, we first need the following elementary result.
Let $f\colon\bR_+\to\bR$ and $g\colon\bR_+\to\bR_+$ be continuous functions, and let $g$ be strictly monotone with $g(0)=0$.
For any $c\in(0,\infty)$, denote
\ba
f_c(t):=f(ct), \quad g_c(t):=g(ct).
\ea
Then, $g_c^{-1}(t)=c^{-1}g^{-1}(t)$, and therefore
\ba
\label{e:fg}
f_c\circ g_c^{-1}=f\circ g^{-1}.
\ea
In particular, \eqref{e:fg} immediately implies that
\ba
(X^{\e,\delta},Y^{\e,\delta} )\circ (A^{\e,\delta,\lambda})^{-1}=(\widetilde X^{\e,\delta},\widetilde Y^{\e,\delta} )\circ (\widetilde A^{\e,\delta,\lambda})^{-1}
\ea
\end{rem}

\begin{cor}
Let conditions of Theorem \ref{thm:limitODE} hold and, additionally, assume that
\ba
\inf_{t,x,y}\gamma(t,x,y)>0.
\ea
If
$\mathfrak{r} =(0,\infty)$, then
\ba
\label{eq:conv_XYA_comp_d}
(X^{\e,\delta},Y^{\e,\delta} )\circ (A^{\e,\delta,\lambda})^{-1}=
(\widetilde X^{\e,\delta},\widetilde Y^{\e,\delta} )\circ (\widetilde A^{\e,\delta,\lambda})^{-1}
\Rightarrow (\widetilde X  ,\widetilde Y  )\circ \widetilde A^{-1}, \quad \e,\delta,\lambda\to 0.
 \ea
In particular, the limit process $(\widehat X, \widehat Y):=(\widetilde X,\widetilde Y  )\circ \widetilde A^{-1}$
satisfies the ODE
\ba
\label{e:XY_1membrane-degen_ODE1}
\widehat X_t
&= x+ \int_0^t
\frac{\beta(0,\widehat X_s,\widehat Y_s)  }{\mathfrak{r}\gamma(0,\widehat  X_s, \widehat Y_s) } \,\di s ,\\
\widehat Y^{i}_t&= y^i+ \int_0^t  \frac{\theta^i(0,\widehat X_s,\widehat Y_s)}{\mathfrak{r}\gamma(0,\widehat X_s, \widehat  Y_s)
}\,\di s,
&\quad i=1,\dots,d.
\ea
If $\mathfrak{r} =\infty$, then
\ba
(X^{\e,\delta},Y^{\e,\delta} )\circ (A^{\e,\delta,\lambda})^{-1}=
(\widetilde X^{\e,\delta},\widetilde Y^{\e,\delta})\circ (\widetilde A^{\e,\delta,\lambda})^{-1} \Rightarrow
(x,y), \quad \e,\delta,\lambda\to 0.
\ea
 \end{cor}
 
\begin{exa}
In this example we illustrate how a suitable arrangement of specially designed membranes can transform
a \emph{free diffusion} into a diffusion with an almost arbitrary drift and diffusion coefficient.

Let $W$ be a standard $(d+1)$-dimensional Brownian motion, and
let $B=(B^i)_{0\leq i\leq d}$ be a sufficiently smooth bounded function on, $\bR^{d+1}$.
For simplicity, consider equidistant membranes located at
$a^\e_k=k\e$, $k\in\bZ$, so that Assumption \textbf{A}$_d^{\mathrm{Lip}_b}$
holds with $d(x)\equiv 1$.

Define the permeability and the sliding coefficients by $\beta:= B^0$ and $\theta^i:=B^i$, $i=1,\dots,d$. Setting
$\e=\delta$, the system with interfaces
\ba
\label{e:XYexa1}
X_t^{\e}&= x+ W^0_t +\sum_{k=-\infty}^\infty \e \int_0^t B^0(k\e,Y_s^{\e})\,\di L^{k\e}_s (X^{\e}),\\
Y^{i,\e}_t&= y^i+W^i_t
+ \sum_{k=-\infty}^\infty \e \int_0^t  B^i(k\e,Y_s^{\e})\,\di L^{k\e}_s (X^{\e}),\quad  i=1,\dots,d,\ t\in[0,\infty),
\ea
converges to the diffusion with drift
\ba
X_t&= x+ W^0_t +\int_0^t B^0(X_s,Y_s)\,\di s,\\
Y_t&= y^i+W^i_t
+ \int_0^t  B^i(X_s,Y_s)\,\di s,\quad  i=1,\dots,d,\ t\in[0,\infty).
\ea
As we see, the semipermeable sliding interfaces do not alter the diffusion coefficient.
This effect was firstly observed in
\cite{Aryasova+24}. A modification of the diffusion coefficient can be achieved by slowing down the diffusion;
however, the new diffusion coefficient will be the same for all coordinates.
Let $C=C(x,y)$
be a sufficiently smooth bounded function with $0<\inf_{x,y}C(x,y)\leq \sup_{x,y}C(x,y)\in(0,1]$.
Defining the permeability, sliding and
stickiness coefficients by
\ba
\gamma(x,y):=  \frac{1}{C^2(x,y)}- 1,
\quad \beta(x,y):= \frac{B^0(x,y)}{C(x,y)^2},\quad
\theta^i(x,y):= \frac{B^i(x,y)}{C^2(x,y)},\quad  i=1,\dots,d.
\ea
Setting $\delta=\lambda=\e$, the diffusion
\ba
\label{e:XYexa2}
X_t^{\e}&= x+ W^0_t +\sum_{k=-\infty}^\infty \e \int_0^t \frac{B^0(k\e,Y_s^{\e})}{C^2(k\e,Y_s^{\e})}\,\di L^{k\e}_s (X^{\e}),\\
Y^{i,\e}_t&= y^i+W^i_t
+ \sum_{k=-\infty}^\infty \e \int_0^t  \frac{B^i(k\e,Y_s^{\e})}{C^2(k\e,Y_s^{\e})}\,\di L^{k\e}_s (X^{\e}),\quad  i=1,\dots,d,\ t\in[0,\infty),
\ea
slowed down by the random time change
\ba
A_t^{\e}
:=t + \e \sum_{k=-\infty}^\infty \int_0^t \Big(\frac{1}{C^2(k\e,Y_s^{\e})}- 1\Big)\,\di L^{k\e}_s (X^{\e}),
\ea
weakly converges to the solution of the SDE
\ba
X_t&= x+ \int_0^t C(X_s,Y_s) \,\di W^0_s +\int_0^t B^0(X_s,Y_s)\,\di s,\\
Y_t&= y^i+  \int_0^t C(X_s,Y_s) \,\di W^i_s
+ \int_0^t  B^i(X_s,Y_s)\,\di s,\quad  i=1,\dots,d,\ t\in[0,\infty).
\ea
\end{exa}

Research on the homogenization of diffusions with interfaces is relatively sparse,
with most existing results restricted to the one-dimensional setting.

Le Gall \cite{LeGall83,legall1984one} studied strong existence and uniqueness for
one-dimensional SDEs involving local times of the unknown process.
Building on the approach of Harrison and Shepp \cite{HShepp-81} for skew Brownian motion,
he transformed such SDEs with local time terms into SDEs with discontinuous coefficients but without local times.
He established strong $L^1$ convergence of the resulting solutions, a result directly applicable to the
homogenization of one-dimensional diffusions with semipermeable interfaces.
A general treatment of these transformations is provided by Engelbert and Schmidt~\cite{engelbert1991strong}; see also Lejay~\cite{Lejay-06}.

Freidlin and Wentzell \cite{Freidlin_et1994} provided conditions for the weak convergence of one-dimensional
Markov processes. In dimension 1, their results recover ours; however, they worked with generalized generators
of Markov processes rather than with solutions of SDEs.

The technique developed by Le Gall \cite{LeGall83,legall1984one} was successfully applied by
Makhno~\cite{makhno2016one,makhno2017diffusion} and Krykun~\cite{krykun2017convergence} to study the
convergence of one-dimensional diffusions with many semipermeable interfaces. They demonstrated,
in particular, that under fairly general conditions on the coefficients,
the presence of membranes with various characteristics may lead to different limiting behaviors: the limit may contain no membranes but
acquire an additional drift term; the membranes may collapse to a single point; they may occupy an entire interval; and so forth.

The proof techniques used in the works mentioned above do not readily extend to the multidimensional case.

In dimension two, Weinryb~\cite{weinryb1984homogeneisation} studied the homogenization of diffusions
with periodic diffusion matrices and semipermeable interfaces located along equally spaced lines or circles,
with periodic penetration probabilities. Her analysis was carried out under the assumption of
existence and uniqueness of solutions to the corresponding martingale problem.

In a series of papers, Hairer and Manson~\cite{hairer2010periodic,hairer2010one,hairer2011multi} considered
multidimensional periodic diffusion homogenization with a single interface,
where the diffusion coefficients are periodic outside a finite-width
interface region. They showed that the scaling limit of such processes is a skew oscillating Brownian motion.

We also mention the works of Ouknine et al.~\cite{trutnau2015countably}
and Ramirez~\cite{ramirez2011multi}, who investigated diffusions with infinitely many interfaces in dimensions one and two.

In \cite{Aryasova+24}, the authors studied a homogenization problem for a system of stochastic differential equations
with local time terms, modeling multivariate diffusion in the presence of semipermeable hyperplane interfaces with oblique penetration. They showed that, as the distances between the interfaces tend to zero,
the singular local-time contributions vanish in the limit and produce an additional
deterministic drift generated by the interfaces, while the diffusion coefficient itself remains unchanged.

\medskip
\noindent
\textbf{Notation.} In this paper,
$|\cdot|$ denotes the Euclidean distance in $\bR^n$, $n\in\mathbb N$,
$\|f\|_\infty=\sup_x|f(x)|$ is the supremum norm of a real-, vector- or matrix-valued function $f$; $x^+=\max\{x,0\}$. 
Sometimes the constant $C\in(0,\infty)$ denotes a generic constant that does not depend on $\e,\delta$ etc;
its value may vary within the same chain of inequalities.

%

We also recall that the symmetric semimartingale local time $L^a(X)$ at $a\in\bR$ of a continuous real-valued semimartingale $X$ is the unique
non-decreasing process satisfying the Tanaka formula
\ba\label{eq:Tanaka_f}
|X_t-a|=|X_0-a|+\int_0^t \sgn (X_s-a)\,\di X_s+ L^a_t(X),
\ea
where
\ba
\sgn x=\begin{cases}
         x/|x|,&\quad x\neq 0,\\
         0,&\quad x=0,\\
        \end{cases}
\ea
see, e.g., Chapter VI in Revuz and Yor \cite{RevuzYor05}.
It is well known that the symmetric semimartingale local time $L^a(X)$ equals the limit
\ba
L^a_t(X)=\lim_{\rho\downarrow 0}\frac{1}{2\rho}\int_0^t \bI(|X_s-a|\leq \rho)\,\di \langle X\rangle_s\quad \text{a.s.}
\ea

\section{Dynamics around one membrane. Important results, estimates. }

Let $\delta\in(0,1]$, and assume that $(X^\delta,Y^\delta, W,(\rF^\delta_t))$ is a weak solution to SDE
\ba
\label{e:XY_1membrane}
X_t^{\delta}&=X_0^{\delta}  +\sum_{l=1}^m \int_0^t \sigma_{l}^0(s, X^{\delta  }_s,Y^{\delta  }_s)\,\di W^l_s
+ \int_0^t b^0(s,X^{\delta }_s,Y_s^{\delta  })\,\di s
+ \delta \int_0^t \beta(s,Y_s^{\delta })\,\di L^{0}_s (X^{\delta }),\\
Y^{i,\delta }_t&= Y^{i,\delta }_0
+\sum_{l=1}^m \int_0^t \sigma^{i}_l(s,X_s^{ \delta},Y_s^{\delta })\,\di W^l_s +\int_0^t b^i(s,X_s^{ \delta},Y_s^{\delta })\,\di s
+  \delta\int_0^t  \theta(s, Y_s^{\delta })\,\di L^{0}_s (X^{ \delta }),\\
&\quad i=1,\dots,d,
\ea
with a single membrane located at the hyperplane $x=0$.
At this stage, the solution may not be unique, and we do not impose any specific assumptions on the coefficients
$\sigma$, $b$, $\beta$ and $\theta$.
Let
$a_-,a_+\in(0,\infty)$. The goal of this Section consists in investigating the
asymptotics of the first exit times of $X^\delta$ from a narrow strip $[-a_-\e,a_+\e]$, $\e\in (0,1]$.
These estimates will later be applied in the analysis of \eqref{e:XY}.
Our argument is subdivided into four steps.

\noindent
\textbf{Step 1.}
For each $\e\in(0,1]$, introduce the filtration $\widetilde \rF_t^{\e,\delta}= \rF^\delta_{\e^2 t}$, the Brownian motion
$\widetilde W^{\e}_t:=\frac{ W_{\e^2 t} }{\e}$, the processes
\ba
&\widetilde X^{\e,\delta}_t:=\e^{-1} X^{\delta }_{\e^2 t},\quad
\widetilde Y^{\e,\delta}_t:=\e^{-1}  Y^{\delta }_{\e^2 t}.
\ea

\begin{lem}
\label{l:31}
The processes $\widetilde X^{\e,\delta}$, $\widetilde Y^{\e,\delta}$ satisfy the SDE
\begin{align}
\label{e:tilde_X}
\widetilde X^{\e,\delta}_t&= \widetilde X^{\e,\delta}_0
+\sum_{l=1}^m \int_0^t \sigma_{l}^0(\e^2 s, \e \widetilde X^{\e,\delta}_{  s}, \e\widetilde Y^{\e,\delta}_{s})\,\di \widetilde W^{l,\e}_s
+ \e\int_0^t b^0(\e^2 s, \e \widetilde X^{\e,\delta}_{  s}, \e \widetilde Y^{\e,\delta}_{s})\,\di s
+ \delta\int_0^t \beta({\e^2 s}, \e \widetilde Y_{s}^{\e,\delta})\,\di L^{0}_{s} (\widetilde X^{\e,\delta}),\\
\label{e:tilde_Y}
\widetilde Y^{i,\e,\delta}_t&= \widetilde Y^{i,\e,\delta}_0
+\sum_{l=1}^m \int_0^t \sigma_{l}^i(\e^2 s, \e \widetilde X^{\e,\delta}_{s}, \e\widetilde Y^{\e,\delta}_{s})\,\di \widetilde W^{l,\e}_s
+ \e\int_0^t b^i(\e^2 s, \e \widetilde X^{\e,\delta}_{  s}, \e \widetilde Y^{\e,\delta}_{s})\,\di s
+ \delta\int_0^t \theta({\e^2 s}, \e \widetilde Y_{s}^{\e,\delta})\,\di L^{0}_{s} (\widetilde X^{\e,\delta}),\\
\notag&\quad i=1,\dots,d,\\
\widetilde X^{\e,\delta}_0&=\e^{-1}X^{\delta}_0,\quad \widetilde Y^{\e,\delta}_0= \e^{-1}Y^{\delta}_0.
\end{align}

\end{lem}
\begin{proof}
We have 
\ba
\label{e:tilde_X1}
\widetilde X^{\e,\delta}_t&
=\widetilde X^{\e,\delta}_0+\frac{1}{\e}\sum_{l=1}^m \int_0^{\e^2 t} \sigma_{l}^0(s, X_{s}^\delta,Y_{s}^\delta)\,\di W^{l}_s
+ \frac{1}{\e}\int_0^{\e^2 t} b^0({s},X_{s}^\delta,Y_{s}^\delta )\,\di s
+\frac{\delta}{\e} \int_0^{\e^2 t} \beta({s},Y_{s}^\delta )\,\di L^{0}_{s} (X^\delta)\\
&=\widetilde X^{\e,\delta}_0+\sum_{l=1}^m \int_0^t \sigma_{l}^0(\e^2s, X_{\e^2  s}^\delta,Y_{\e^2s}^\delta)\,\di \frac{W^{l}_{\e^2 s}}{\e}
+ \e\int_0^t b^0({\e^2s},X_{\e^2s}^\delta,Y_{\e^2s}^\delta )\,\di s
+\frac{\delta}{\e} \int_0^t \beta(\e^2s,Y_{\e^2s}^\delta )\,\di L^{0}_{\e^2s} (X^\delta )\\
&=\widetilde X^{\e,\delta}_0+\sum_{l=1}^m \int_0^t \sigma_{l}^0(\e^2 s, \e \widetilde X^{\e,\delta}_{  s}, \e\widetilde Y^{\e,\delta}_{s})
\,\di \widetilde W^{l,\e}_s
+ \e\int_0^t b^0(\e^2 s, \e \widetilde X^{\e,\delta}_{  s}, \e \widetilde Y^{\e,\delta}_{s})\,\di s
+ \frac{\delta}{\e}\int_0^t \beta({\e^2 s}, \e \widetilde Y_{s}^{\e,\delta})\,\di L^{0}_{\e^2s} (X^\delta ).
\ea
Further,
\ba
\label{e:localtime_tilde_X}
L^{0}_{\e^2t} (X^\delta )
&=\lim_{\rho\downarrow 0}\frac{1}{2\rho}\int_0^{\e^2t}\bI(|X_s^\delta|\leq \rho)\,\di \langle X^\delta\rangle_s\\
&=\e^2\lim_{\rho\downarrow 0}\frac{1}{2\rho}\int_0^t\bI(|X_{\e^2 s}^\delta|\leq \rho)
\Sigma^{00}(\e^2 s, X^\delta_{\e^2  s},Y^\delta_{\e^2 s})\,\di s\\
&=\e\lim_{\rho\downarrow 0}\frac{\e}{2\rho}
\int_0^t \bI\Big(\Big|\frac{X^\delta_{\e^2 s}}{\e}\Big|\leq \frac{\rho}{\e}\Big)
\Sigma^{00}(\e^2 s, X^{ \delta}_{\e^2  s},Y^{ \delta}_{\e^2 s})\,\di s\\
&=\e L_t^0(\widetilde X^{\e,\delta}).
\ea
Substituting \eqref{e:localtime_tilde_X} into \eqref{e:tilde_X1} we obtain \eqref{e:tilde_X}.
The second equation is obtained similarly.
\end{proof}

\noindent
\textbf{Step 2.} For $a_-,a_+\in(0,\infty)$,
introduce
the stopping times
\ba
\tau^{\e,\delta} & =\inf\{t\in[ 0,\infty)\colon X_t^\delta\notin (-\e a_{- }, \e a_+)\}, \\
\widetilde \tau^{\e,\delta} & =\inf\{t\in[0,\infty)\colon \widetilde X^{\e,\delta}_t\notin ( -a_{- },  a_+)\}.
\ea

\begin{lem}
\label{l:32}
The following relations hold:
\ba
1.\quad &\widetilde \tau^{\e,\delta}=\frac{1}{\e^2}\tau^{\e,\delta},\\
2.\quad &\widetilde X^{\e,\delta}(\widetilde \tau^{\e,\delta})=\frac{1}{\e}X^\delta(\tau^{\e,\delta}),\\
3.\quad &
\P (\widetilde X^{\e,\delta} (\widetilde \tau^{\e,\delta})=\pm a_{\pm})
=\P (X^\delta(\tau^{\e,\delta})=\pm\e a_{\pm }),\\
4.\quad &L^0_{\widetilde \tau^{\e,\delta}}(\widetilde X^{\e,\delta})=\frac{1}{\e}  L^0_{\tau^{\e,\delta}}(  X^\delta).
\ea
\end{lem}
\begin{proof}
To show 1.\ we note that
\ba
\widetilde \tau^{\e,\delta}
&=\inf\{t\geq 0\colon \widetilde X^{\e,\delta}_t\notin (-a_{- },  a_+)\}
 =\inf\Big\{t\geq 0\colon  \frac{1}{\e}X_{\e^2 t}^\delta\notin (-a_{- },  a_+)\Big\}\\
 &=\inf\Big\{\frac{s}{\e^2}\geq 0\colon  X_{s}^\delta\notin ( -\e a_{- },  \e a_+)\Big\}
 =\frac{1}{\e^2}\inf\Big\{{s}\geq 0\colon  X_{s}^\delta\notin (-\e a_{- },  \e a_+)\Big\}=\frac{1}{\e^2}\tau^{\e,\delta}.
\ea
From this formula we get 2.:
\ba
\widetilde X^{\e,\delta}(\widetilde \tau^{\e,\delta})
=\frac{1}{\e}X^\delta(\e^2 \widetilde \tau^{\e,\delta})=\frac{1}{\e}X^\delta(\tau^{\e,\delta}),
\ea
and 3.\ follows immediately. Eventually, relation 4.\ follows from
the definition of a local time, see \eqref{e:localtime_tilde_X}.
\end{proof}

\noindent
\textbf{Step 3.} We give rough and general estimates on the exit times $\tau^{\e,\delta}$ and
$\widetilde\tau^{\e,\delta}$
from the strip, local times, and the deviation from the
initial point.

\begin{lem}
\label{cor:estimate_hittings}
 Assume that there is $C\in[1,\infty)$
 such that
\ba
\label{e:aC}
 a_\pm\in[C^{-1}, C]
\ea
and for all $t\in[0,\infty)$, $(x,y)\in [-a_-,a_+]\times \bR^d$
\ba
\label{eq:ass_bounds_coeff_SDE}
\sum_{l=1}^m|\sigma_l^0(t,x,y)|+ |b^0(t,x,y)|\leq C, \quad
\Sigma^{00}(t,x,y)\geq C^{-1}.
\ea
Then, there are constants $A_k\in(0,\infty)$,  $k\in\mathbb N$, and a constant $B\in(0,\infty)$
that depend only on $C$, such that for any $\e,\delta\in(0,1]$
\begin{align}
\label{e:X1}
&
\E \Big[|\widetilde X^{\e,\delta}_{\widetilde\tau^{\e,\delta}}-\widetilde X^{\e,\delta}_0|^k
\Big|\widetilde  \rF_{0}^{\e,\delta}\Big]\leq A_k,
&&
\E \Big[|  X^{\delta}( \tau^{\e,\delta})-X^\delta_0|^k  \Big|\rF^{\delta}_{0}\Big]\leq A_k\e^k,
\quad \text{a.s.}\\
\label{e:Y1}
&
\E \Big[\sup_{s\in[0,\widetilde \tau^{\e,\delta}]}|\widetilde Y^{\e,\delta}_s-\widetilde Y^{\e,\delta}_0|^k
\Big|\widetilde \rF_{0}^{\e,\delta}\Big]
\leq A_k,
&&
\E \Big[\sup_{s\in [0, \tau^{\e,\delta}]}| Y^\delta_s-Y_0^\delta|^k  \Big|\rF_{0}^\delta\Big]\leq A_k\e^k,
\quad \text{a.s.}\\
\label{e:L1}
&
\E \Big[   ( L^0_{\widetilde\tau^{\e,\delta} }(\widetilde X^{\e,\delta}))^k \Big|\widetilde\rF_{0}^{\e,\delta}\Big]\leq A_k,
&&
\E \Big[  (L_{\tau^{\e,\delta}}^0)^k  \Big|\rF_{0}^\delta\Big]\leq A_k\e^k, \quad \text{a.s.}\\
\label{e:tau1}
&
\E \Big[ ( \widetilde\tau^{\e,\delta})^k  \Big|\widetilde \rF_0^{\e,\delta}\Big]\leq A_k,
&&
\E \Big[  (\tau^{\e,\delta})^k  \Big|\rF_0^\delta\Big]\leq A_k\e^{2k},\quad \text{a.s.}\\
\intertext{and if $X^\delta_0=0$}
&\label{e:tau2}
\E \Big[  \widetilde\tau^{\e,\delta}  \Big|\widetilde \rF_0^{\e,\delta}\Big]\geq B ,
&&
\E \Big[  \tau^{\e,\delta}  \Big|\rF_0^\delta\Big]\geq B \e^2,
\quad \text{a.s.}
\end{align}
\end{lem}

%
%
%
%
%
\begin{proof}
We prove the formulas for the processes $\widetilde X^{\e,\delta}$ and  $\widetilde Y^{\e,\delta}$; their
counterparts for $ X^{\delta}$ and  $Y^{\delta}$ follow from Lemma \ref{l:32}. For simplicity, we omit conditional expectations.

The inequality \eqref{e:X1} is obvious for any $A_k\geq C^k\geq  (a_-\vee a_+)^k$.

For some $\rho,A,K\in(0,\infty)$ to be chosen later
let us consider the Lyapunov function
\ba
h(t,x)=\ex^{\rho t}\big(A-\cosh(Kx)\big).
\ea
Then $h(0,x)\leq A$ and
\ba
\partial_t h(t,x)&= \rho\ex^{\rho t}\big(A-\cosh(Kx)\big),\\
\partial_x h(t,x)&=-K\ex^{\rho t}\sinh(Kx), \quad  \partial_x h(t,0)=0, \\
\partial_{xx} h(t,x)&=-K^2\ex^{\rho t}\cosh(Kx) .
\ea    
The application of the It\^o formula to the stopped process $\widetilde X^{\e,\delta}_{t\wedge N}$,
$N\in\mathbb N$, starting at $\frac{x}{\e}$ yields
\ba
\E h(\widetilde \tau^{\e,\delta}\wedge N,\widetilde X^{\e,\delta}_{\widetilde \tau^{\e,\delta}\wedge N})
&=h\Big(0,\frac{x}{\e}\Big)
+\E\int_0^{\widetilde \tau^{\e,\delta}\wedge N} \rho
\ex^{\rho s} \Big( A-\cosh(K\widetilde X^{\e,\delta}_s)\Big)\,\di s\\
&-\E\int_0^{\widetilde \tau^{\e,\delta}\wedge N}K b^0(\e^2 s, \e \widetilde X^{\e,\delta}_{  s}, \e\widetilde Y^{\e,\delta}_{s})
\ex^{\rho s}\sinh(K\widetilde X^{\e,\delta}_s)\,\di s\\
&-\frac12 \E\int_0^{\widetilde \tau^{\e,\delta}\wedge N}
\Sigma^{00}(\e^2 s, \e \widetilde X^{\e,\delta}_{  s}, \e\widetilde Y^{\e,\delta}_{s})
K^2\ex^{\rho s} \cosh(K\widetilde X^{\e,\delta}_s)  \,\di s.
\ea
Since $\cosh x\geq 1$ and $|\sinh x|\leq \cosh x$, it follows from \eqref{eq:ass_bounds_coeff_SDE}   that
\ba
\rho
\Big( A-\cosh(K\widetilde X^{\e,\delta}_s)\Big)
&-K b^0(\e^2 s, \e \widetilde X^{\e,\delta}_{  s}, \e\widetilde Y^{\e,\delta}_{s})
\sinh(K\widetilde X^{\e,\delta}_s)
-\frac12 \Sigma^{00}(\e^2 s, \e \widetilde X^{\e,\delta}_{  s}, \e\widetilde Y^{\e,\delta}_{s})
K^2\cosh(K\widetilde X^{\e,\delta}_s)\\
&\leq \rho A+\Big(K\|b\| -\frac{K^2}2 \Sigma^{00}(\e^2 s, \e \widetilde X^{\e,\delta}_{  s}, \e\widetilde Y^{\e,\delta}_{s})\Big) \cosh(K\widetilde X^{\e,\delta}_s)\\
&\leq
  \rho A+K\Big(C -\frac{K}{2C}\Big) \cosh(KC).
\ea
First, choose $K=K(C)$ large enough such that
\ba
K\Big(C -\frac{K}{2C}\Big) \cosh(KC)\leq -2.
\ea
Then, choose $A=A(C,K)\geq 1$ large enough such that $A - \cosh(KC)\geq 0$, so that
$h(\widetilde \tau^{\e,\delta}\wedge N,\widetilde X^{\e,\delta}_{\widetilde \tau^{\e,\delta}\wedge N})\geq 0$.
Finally, choose $\rho=\rho(C,K,A)>0$ small enough such that
\ba
\rho A+K\Big(C -\frac{K}{2C}\Big) \cosh(KC)\leq -1.
\ea
Then 
\ba
0&\leq \E h(\widetilde \tau^{\e,\delta}\wedge N,\widetilde X^{\e,\delta}_{\widetilde \tau^{\e,\delta}\wedge N})
\leq h(0,\frac{x}{\e})- \E \int_0^{\widetilde\tau^{\e,\delta}\wedge N}\ex^{\rho s}\,\di s\\
&\leq A-\frac {1}{\rho} \big(\E\ex^{\rho(\widetilde\tau^{\e,\delta}\wedge N)}-1\big),\\
\ea
Therefore,
\ba
\E \ex^{\rho(\widetilde\tau^{\e,\delta}\wedge N)}&\leq  A\rho +1, \ N\in\mathbb N.
\ea
Passing to the limit as $N\to\infty$, we obtain
\ba
\label{e:Moment Z}
\E \ex^{\rho \widetilde\tau^{\e,\delta}}\leq A\rho +1
\ea
and therefore
\ba
\E (\widetilde\tau^{\e,\delta})^k \leq \frac{k!}{\rho^k} \E \ex^{\rho \widetilde\tau^{\e,\delta}}
\leq\frac{k!}{\rho^k} (A\rho +1)
\ea
and \eqref{e:tau1} follows.

Since $\sgn 0=0$ we get
\ba
\label{e:0}
\delta\int_0^t \sgn (\widetilde X^{\e,\delta}_s) \beta({\e^2 s}, \e \widetilde Y_{s}^{\e,\delta})\,\di L^{0}_{s} (\widetilde X^{\e,\delta})
=0.
\ea
Applying Tanaka's formula, \eqref{e:tilde_X} and \eqref{e:0}
we obtain
\ba
L^0_t(\widetilde X^{\e,\delta})
&=
|\widetilde X^{\e,\delta}_t|-|\widetilde X^{\e,\delta}_0|-\int_0^t \sgn (\widetilde X^{\e,\delta}_s)\,\di \widetilde X^{\e,\delta}_s\\
&=
|\widetilde X^{\e,\delta}_t|-|\widetilde X^{\e,\delta}_0|
-\sum_{l=1}^m \int_0^t \sgn (\widetilde X^{\e,\delta}_s)
\sigma_{l}^0(\e^2 s, \e \widetilde X^{\e,\delta}_{  s}, \e\widetilde Y^{\e,\delta}_{s})\,\di \widetilde W^{l,\e}_s\\
&- \e\int_0^t \sgn (\widetilde X^{\e,\delta}_s) b^0(\e^2 s, \e \widetilde X^{\e,\delta}_{  s}, \e \widetilde Y^{\e,\delta}_{s})\,\di s.
\ea
Then by the BDG inequality, \eqref{e:X1}, \eqref{e:tau1} and the boundedness of all the coefficients we get \eqref{e:L1}.

The estimate \eqref{e:Y1} follows analogously with the help of the representation
\eqref{e:tilde_Y}, the boundnedness of all the coefficients, and  \eqref{e:tau1} and \eqref{e:L1}.

By the It\^o formula, from \eqref{e:tilde_X} and $x=0$ we get
\ba
(a_-\wedge a_+)^2&\leq
\E (\widetilde X^{\e,\delta}_{ \widetilde \tau^{\e,\delta} }  )^2
=  \E \int_0^{\widetilde \tau^{\e,\delta} }
\Sigma^{00}(\e^2 s, \e \widetilde X^{\e,\delta}_{  s}, \e\widetilde Y^{\e,\delta}_{s})\,\di s
+ 2\e\E \int_0^{\widetilde \tau^{\e,\delta} }
\widetilde X^{\e,\delta}_s b^0(\e^2 s, \e \widetilde X^{\e,\delta}_{  s}, \e \widetilde Y^{\e,\delta}_{s})\,\di s\\
&\leq (\|\Sigma^{00}\| + 2\e (a_-\vee a_+) \|b\| ) \E  \widetilde \tau^{\e,\delta},
\ea
so that \eqref{e:tau2} holds true.
\end{proof}

%
%

\noindent
\textbf{Step 4.} Asymptotic expansion of probabilistic characteristics of $ X^{\e,\delta}$.

\begin{lem}\label{lem:accu_estimates}
Assume that \eqref{e:aC} and \eqref{eq:ass_bounds_coeff_SDE}
hold,
and $\widetilde X_0^{\e,\delta}=0$ and $\widetilde Y_0^{\e,\delta}=\frac{y}{\e}$. Then,
\begin{align}
\label{e:estP}
\P\Big(\widetilde X_{ \tau^{\e,\delta}}^{\e,\delta}=a_\pm\Big|\widetilde \rF^{\e,\delta}_0\Big)&=\frac{   a_\mp}{a_-+a_+}
+\mathcal O(\e)+\mathcal O(\delta),  \quad \e, \delta\to 0,\\
\label{e:estL}
\E \Big[ L^0_{\widetilde \tau^{\e,\delta}}(\widetilde X^{\e,\delta})\Big|\widetilde \rF^{\e,\delta}_0\Big]
&  =  \frac{ 2 a_-a_+}{a_-+a_+} +\mathcal O(\e)+ \mathcal O(\delta), \quad \e, \delta\to 0,
\end{align}
and if $(t,x,y)\mapsto \Sigma^{00}(t,x,y)$, $b^0(t,x,y)$, $\beta(t,y)$ are globally Lipschitz continuous
\begin{align}
\label{e:esttau}
\E \Big[\widetilde\tau^{\e,\delta}\Big|\widetilde \rF^{\e,\delta}_0\Big]
&= \frac{a_-a_+}{\Sigma^{00}(0,0,y)} +\mathcal O(\e)+ \mathcal O(\delta), \quad \e, \delta\to0,\\
\E \Big[\widetilde X_{\widetilde \tau^{\e,\delta}}^{\e,\delta}\Big|\widetilde \rF^{\e,\delta}_0\Big]
& =   \frac{b^0(0,0,y)}{\Sigma^{00}(0,0,y)}a_-a_+  \e +\beta(0,y)\frac{2 a_-a_+}{a_-+a_+}\delta
 +\mathcal O(\e^2)+\mathcal O(\e\delta), \quad \e, \delta\to0.
\end{align}
where all error terms $\mathcal O=\mathcal O_{\omega,\e,\delta}$ are such that
\ba
\operatorname*{ess\,sup}_{\omega}\sup_{\e,\delta}
\Big( \frac{|\mathcal O(\e)|}{\e}
+ \frac{|\mathcal O(\e^2)|}{\e^2}
+ \frac{|\mathcal O(\delta)|}{\delta}
+ \frac{|\mathcal O(\e\delta)|}{\e\delta} \Big)
\ea
is bounded by a non-random constant that depends only on $C$ from \eqref{e:aC} and \eqref{eq:ass_bounds_coeff_SDE}
and on the Lipschitz constant $C_\mathrm{Lip}$ of all functions.
\end{lem}
\begin{proof}
1.  For simplicity, we omit conditional expectations. It follows from Lemma \ref{cor:estimate_hittings} that
\ba
\E \widetilde X_{\widetilde \tau^{\e,\delta}}^{\e,\delta}
&=  \e \E \int_0^{ \widetilde\tau^{\e,\delta}} b^0(\e^2 s, \e \widetilde X^{\e,\delta}_{  s}, \e \widetilde Y^{\e,\delta}_{s})\,\di s
  + \delta\E
  \int_0^{\widetilde \tau^{\e,\delta}} \beta({\e^2 s},  \e \widetilde Y_{s}^{\e,\delta})\,\di L^{0}_{s} (\widetilde X^{\e,\delta})\\
  &=\e \mathcal O(1)+\delta \mathcal O(1)=  \mathcal O(\e)+ \mathcal O(\delta).
\ea
On the other hand
\ba
\E \widetilde X_{\widetilde \tau^{\e,\delta}}^{\e,\delta}
&=-a_- \P( \widetilde X_{\widetilde \tau^{\e,\delta}}^{\e,\delta}=a_-)
+ a_+ \P( \widetilde X_{\widetilde \tau^{\e,\delta}}^{\e,\delta}=a_+)\\
&= - a_- (1-\P( \widetilde X_{\widetilde \tau^{\e,\delta}}^{\e,\delta}=a_+))
+ a_+ \P( \widetilde X_{ \widetilde\tau^{\e,\delta}}^{\e,\delta}=a_+).
\ea
This implies the estimate \eqref{e:estP}.

2. Tanaka's formula \eqref{eq:Tanaka_f} implies that
\ba
\E  L^0_{\widetilde \tau^{\e,\delta}}(\widetilde X^{\e,\delta})
&=\E| \widetilde X_{\widetilde \tau^{\e,\delta}}^{\e,\delta}|
- \E \int_0^{\widetilde \tau^{\e,\delta}} \sgn (\widetilde X_{s}^{\e,\delta})\,\di \widetilde X_{s}^{\e,\delta}\\
&=a_- \P( \widetilde X_{\widetilde \tau^{\e,\delta}}^{\e,\delta}=-a_-)+ a_+ \P( \widetilde X_{\widetilde
\tau^{\e,\delta}}^{\e,\delta}=a_+)
-  \e \E \int_0^{ \widetilde\tau^{\e,\delta}} \sgn (\widetilde X_{s}^{\e,\delta}) b^0(\e^2 s, \e \widetilde X^{\e,\delta}_{  s}, \e \widetilde Y^{\e,\delta}_{s})\,\di s\\
&=a_- \frac{   a_+}{a_-+a_+} + a_+ \frac{   a_-}{a_-+a_+}
+ \mathcal O(\e)+ \mathcal O(\delta)+ \e \mathcal O(1 )\E \widetilde\tau^{\e,\delta}\\
&= \frac{2a_-    a_+}{a_-+a_+}  + \mathcal O(\e)+ \mathcal O(\delta).
\ea
3.
We have 
\ba
\E ( \widetilde X_{\widetilde \tau^{\e,\delta}}^{\e,\delta})^2
&=a_-^2 \P( \widetilde X_{\widetilde \tau^{\e,\delta}}^{\e,\delta}=-a_-)
+ a_+^2 \P( \widetilde X_{ \widetilde\tau^{\e,\delta}}^{\e,\delta}=a_+)\\
&= a_-^2 \frac{   a_+}{a_-+a_+} + a_+^2 \frac{   a_-}{a_-+a_+} +\mathcal  O(\e)+\mathcal  O(\delta)\\
&=2a_- {   a_+}  + \mathcal O(\e)+ \mathcal O(\delta).
\ea
On the other hand it follows from the It\^o formula that
\ba
\E (\widetilde X^{\e,\delta}_{ \widetilde \tau^{\e,\delta} }  )^2
&=  \E \int_0^{\widetilde \tau^{\e,\delta} }
\Sigma^{00}(\e^2 s, \e \widetilde X^{\e,\delta}_{  s}, \e\widetilde Y^{\e,\delta}_{s})\,\di s
+ 2\e\E \int_0^{\widetilde \tau^{\e,\delta} }
\widetilde X^{\e,\delta}_s b^0(\e^2 s, \e \widetilde X^{\e,\delta}_{  s}, \e \widetilde Y^{\e,\delta}_{s})\,\di s\\
&=
\E \int_0^{\widetilde \tau^{\e,\delta} }
\Sigma^{00}(\e^2 s, \e \widetilde X^{\e,\delta}_{  s}, \e\widetilde Y^{\e,\delta}_{s})\,\di s
+ \mathcal O(\e)\\
&=
\Sigma^{00}(0,0,y) \E \widetilde \tau^{\e,\delta}
+ \E \int_0^{\widetilde \tau^{\e,\delta} }
(\Sigma^{00}(\e^2 s, \e \widetilde X^{\e,\delta}_{  s}, \e\widetilde Y^{\e,\delta}_{s})-\Sigma^{00}(0,0,y) )\,\di s
+ \mathcal O(\e)\\
\ea
By the Lipschitz property of $\Sigma^{00}$ we have
\ba
\label{e:340}
&\E \int_0^{\widetilde \tau^{\e,\delta} }
\Big|\Sigma^{00}(\e^2 s, \e \widetilde X^{\e,\delta}_{  s}, \e\widetilde Y^{\e,\delta}_{s})-\Sigma^{00}(0,0,y)\Big|\,\di s\\
&\leq C_\mathrm{Lip} \E \int_0^{\widetilde \tau^{\e,\delta} }
\sup_{u\in[0,\widetilde \tau^{\e,\delta} ]}\Big(|\e^2 u|+ |\e \widetilde X^{\e,\delta}_{u}|
+ \e|\widetilde Y^{\e,\delta}_u-\frac{y}{\e}|\Big)\,\di s=\mathcal O(\e),
\ea
see Lemma \ref{cor:estimate_hittings}, and the estimate \eqref{e:esttau} follows.

4. Taking the expectation in \eqref{e:tilde_X} we get
\ba
\E \widetilde X^{\e,\delta}_{\widetilde \tau^{\e,\delta} }
&=
\e\E \int_0^{\widetilde \tau^{\e,\delta} } b^0(\e^2 s, \e \widetilde X^{\e,\delta}_{  s}, \e \widetilde Y^{\e,\delta}_{s})\,\di s
+ \delta\E
\int_0^{\widetilde \tau^{\e,\delta} } \beta({\e^2 s}, \e \widetilde Y_{s}^{\e,\delta})\,\di L^{0}_{s} (\widetilde X^{\e,\delta})\\
&=
\e b^0(0, 0, y)\E \widetilde \tau^{\e,\delta}
+ \delta \beta(0,y)  \E L^{0}_{\widetilde \tau^{\e,\delta}} (\widetilde X^{\e,\delta}) \\
&+
\e\E \int_0^{\widetilde \tau^{\e,\delta} }
(b^0(\e^2 s, \e \widetilde X^{\e,\delta}_{  s}, \e \widetilde Y^{\e,\delta}_{s})- b^0(0, 0, y))\,\di s
+ \delta\E
\int_0^{\widetilde \tau^{\e,\delta} }
( \beta({\e^2 s}, \e \widetilde Y_{s}^{\e,\delta}) -\beta(0,y) )\,\di L^{0}_{s} (\widetilde X^{\e,\delta})\\
&=\e b^0(0, 0, y)\E \widetilde \tau^{\e,\delta}
+ \delta \beta(0,y)  \E L^{0}_{\widetilde \tau^{\e,\delta}} (\widetilde X^{\e,\delta})
+ \mathcal O(\e^2)+\mathcal O(\e\delta)
\ea
as in \eqref{e:340}. The proof is finished.
\end{proof}

The next corollary follows immediately from Lemma \ref{l:32} and Lemma \ref{lem:accu_estimates}.

\begin{cor}
\label{cor:accurate_estim_X}
Under the assumptions of Lemma \ref{lem:accu_estimates}, the following estimates hold:
\ba
\label{eq:accurate_estim}
\E \Big[ L^0_{ \tau^{\e,\delta}}(X^{\delta})\Big|\rF_0^\delta\Big]
&  =  \frac{ 2 a_-a_+}{a_-+a_+}\e +\mathcal O(\e^2)  + \mathcal O(\e\delta) ,\\
\E \Big[  \tau^{\e,\delta}\Big|\rF_0^\delta\Big] &= \frac{a_-a_+}{\Sigma^{00}(0,0,y)} \e^2 +\mathcal O(\e^3)  + \mathcal O(\e^2\delta),\\
\E  \Big[ X_{ \tau^{\e,\delta}}^{\e} \Big|\rF_0^\delta\Big]& =
 \frac{b^0(0,0,y)}{\Sigma^{00}(0,0,y)}a_-a_+\e^2 +\beta(0,y)\frac{2 a_-a_+}{a_-+a_+}\e\delta
 +\mathcal O(\e^3)+\mathcal O(\e^2\delta).
\ea
\end{cor}

\section{Proof of Theorem \ref{t:existU}\label{s:e}}

\subsection{Existence. Proof of Theorem \ref{t:existU} 1)}

The proof of existence is done by compactness argument.

We start with a general result about weakly relative compactness of SDEs with local times.

\begin{thm}
\label{t:Zncomp}
Let for each $n\in\mathbb N$, $(W^n_t,\rF^n_t)_{t\in[0,\infty)}$ be a standard $d$-dimensional Wiener process, and
let $\sigma_{l}^n$, $b^n$, $\beta^n$ be $\rF_t^n$-adapted c\`adl\`ag processes
such that there is $C\in(0,\infty)$ such that
\ba
\sup_{n\in\mathbb N}\sup_{t\in [0,\infty)}\Big(\|\sigma^n(t)\|+|b^n(t)|+|\beta^n(t)|\Big)\leq C
\ea
with probability 1.
Assume that $\{Z^n\}_{n\in\mathbb N}$ is a family of real valued It\^o processes
\ba
\label{e:Zn}
 Z^n_t=Z^n_0 + \sum_{l=1}^d \int_0^t \sigma_l^n (s)\, \di W^{l,n}_s
+  \int_0^t b^n(s)\,\di s +  \int_0^t \beta^n(s)\, \di L^{0}_s(Z^n),
\ea
where $L^{0}$ is a symmetric semimartingale local time at $0$.
Let $\{Z_0^n\}_{n\in\mathbb N}$ be weakly relative compact family of initial conditions.
Then the sequence  $\{Z^n,  L^{0}(Z^n)\}_{n\in\mathbb N}$ is weakly relative compact
in $C(\bR_+,\bR\times \bR_+)$.
\end{thm}
\begin{proof}
 The weak relative compactness of stochastic and Lebesgue integrals
 $\{\int_0^\cdot \sigma_{l}^n(s)\, \di W^{l,n}_s \}$,  $\{\int_0^\cdot b^n(s)\,\di s\}$
 is due to the uniform boundedness of the coefficients $\sigma^n$ and $b^n$.

Applying the It\^o formula to $(Z^n)^2$, we cancel the term with the local time and get
 \ba
(Z^n_t)^2=(Z^n_0)^2
+ \sum_{l=1}^d \int_0^t Z^n_s\sigma_l^n (s)\, \di W^{l,n}_s
+  2 \int_0^t Z^n_s b^n(s)\,\di s
+\sum_{l=1}^d \int_0^t (\sigma_l^n (s))^2\, \di s.
\ea
Since the integrands satisfy the linear growth condition, the weak relative compactness of $\{(Z^n)^2\}$ follows
a standard argument. In particular, we obtain the weak relative compactness of $\{|Z^n|\}$.

Applying the Tanaka formula to $|Z^n|$ we get
\ba
|Z^n_t| &=|Z^n_0|+ \sum_{l=1}^d \int_0^t \sgn(Z^n_s)\sigma_l^n (s)\, \di W^{l,n}_s
+  \int_0^t \sgn(Z^n_s) b^n(s)\,\di s   + L^0_t(Z^n).
\ea
Since the stochastic and Lebesgue integrals are weakly relatively compact due to the uniform boundedness of the integrands,
we get the weak relative compactness of $\{L^0(Z^n)\}$.

Eventually, since $\beta^n$ are uniformly bounded, we have the w.r.c.\ of
the integrals
$\{\int_0^\cdot  \beta^n(s)\, \di L^{0}_s(Z^n)\}$.

Let now $Z$ be a condensation point of $\{Z^n\}$. Passing to a subsequence if necessary,
we apply
Proposition 1.17 p.\ 526 in \cite{JacodS-03}
to the sequence of stochastic integrals to get that its limit is a continuous local martingale.
The Lebesgue integrals converge to a process of bounded variation, as well as the integrals w.r.t.\ $L^0(Z^n)$.
\end{proof}

\begin{thm}
\label{t:Znconv}
Assume that the family $\{Z^n\}$ satisfies the conditions of Theorem \ref{t:Zncomp}.
Let $(W_t,\rF_t)_{t\geq 0}$ be a standard $d$-dimensional Wiener process, and
let $Z$, $\sigma_l$, $b$, $\beta$ be continuous $\rF_t$-adapted processes.

Assume that for any $T\in[0,\infty)$,
\ba
\label{e:conv1}
\sup_{t\in[0,T]}\Big( | Z^n_t - Z_t|+\|\sigma^n(t)-\sigma(t)\|+
|b^n(t)-b(t)|+ |\beta^n(t)-\beta (t)|+ \|W^{n}_t-W_t\| \Big)\to 0, \ n\to\infty,
\ea
with probability 1, and for any $t\in [0,\infty)$
\ba
\label{e:sigma}
\sum_{l=1}^d(\sigma_l(t))^2>0
\ea
with probability 1.
Then the process
$Z$ satisfies the equation
\ba
Z_t=Z_0 + \sum_{l=1}^d \int_0^t \sigma_{l}(s)\, \di W^{l}_s +  \int_0^t b(s)\,\di s +  \int_0^t \beta(s)\, \di L^{0}_s( Z).
\ea
\end{thm}
\begin{proof}
1. Convergence of the Lebesgue integrals follows from the dominated convergence theorem.

2.
The local uniform convergence in probability of stochastic integrals
follows from Theorem 2.2 in \cite{KurtzP-91}.

3. From \eqref{e:Zn} we can represent the integrals w.r.t.\ $L^0(Z^n)$ as
\ba
\int_0^t \beta^n(s)\, \di L^{0}_s(Z^n)
=Z^n_t-Z^n_0 - \sum_{l=1}^d \int_0^t \sigma_l^n (s)\, \di W^{l,n}_s
-\int_0^t b^n(s)\,\di s.
\ea
It follows from 1., 2.\ and \eqref{e:conv1} that sequence $\{\int_0^\cdot \beta^n(s)\, \di L^{0}_s(Z^n)\}$
is locally uniformly convergent in probability. Its limit is a process of locally bounded variation. Indeed,
for any $T\in[0,\infty)$ we have
\ba
\mathrm{Var}_{[0,T]}\Big(\int_0^\cdot\beta^n(s)\, \di L^{0}_s( Z^n)\Big)\leq C  L^{0}_T( Z^n).
\ea
It follows from Theorem \ref{t:Zncomp} that the family $\{L_T^0(Z^n)\}$ is bounded in probability, therefore
the sequence of variations $\mathrm{Var}_{[0,T]}\Big(\int_0^\cdot\beta^n(s)\, \di L^{0}_s( Z^n)\Big)$ is also
bounded in probability, and the statement follows.

4. The limit process $Z$ is a continuous semimartingale due to Proposition 1.17 p.\ 526 in \cite{JacodS-03} and the fact that the limits
of Lebesgue integrals and integrals w.r.t.\ $L^0(Z^n)$ are processes of locally bounded variation.
Its martingale component equals to
\ba
M_t:=\sum_{l=1}^d \int_0^t \sigma_{l}(s)\, \di W^{l}_s
\ea
with the bracket
\ba
\langle M_t\rangle:=\sum_{l=1}^d \int_0^t |\sigma_{l}(s)|^2\, \di s.
\ea
Due to assumption \eqref{e:sigma} and the occupation times formula
\ba
\label{e:zero}
\int_0^\infty \bI(Z_t=0)\,\di t=0 \quad\text{a.s.}
\ea

3. It is left to establish convergence of integrals w.r.t.\ the local times. To do this, we prove convergence of the local times.
The Tanaka formula yields:
\ba
\label{e:LL}
 L^0_t(Z^n)
 =|Z^n_t|-|Z^n_0|
 - \sum_{l=1}^d \int_0^t \sgn(Z^n_s)\sigma_l^n (s)\, \di W^{l,n}_s
-  \int_0^t \sgn(Z^n_s) b^n(s)\,\di s.
\ea
From \eqref{e:zero}, it follows for almost all $s\in[0,\infty)$, that
\ba
\P\Big(\lim_{n\to\infty}\sgn(Z^n_s) = \sgn(Z_s)\Big)=1.
\ea
In \eqref{e:LL}, the Lebesgue integrals converge a.s.\ by the dominated convergence theorem.

It is left to demonstrate the convergence of the stochastic integrals. Since $s\mapsto\sgn(Z_s)$
is not c\`adl\`ag, the standard convergence results as, e.g., Theorem 2.2 in \cite{KurtzP-91}, cannot be applied directly.
Therefore we now prove this convergence.

Let $l=1,\dots,d$ be fixed.
We also fix an arbitrary $\e\in(0,1]$ and find $N\in\mathbb N$ and time instants
$0=t_0<t_1\dots<t_N\leq t$ such that
with probability one $Z_{t_k}\neq 0$ for all $k=1,\dots,N$, and the following approximation holds:
\ba
\label{e:eps}
\E \int_0^t \Big|\sgn(Z_s)\sigma_l (s) - \sum_{k=1}^{N-1} \sgn(Z_{t_k})\sigma_l (t_k)\bI_{[t_k,t_{k+1})}(s)\Big|^2\,\di s<\e.
\ea
We have the following elementary estimate:
\ba
\label{e:eps1}
\E\Big|\int_0^t & \sgn(Z^n_s)\sigma_l^n (s)\, \di W^{l,n}_s - \int_0^t \sgn(Z_s)\sigma_l (s)\, \di W^{l}_s\Big|^2\\
&\leq 3 \E\Big|\int_0^t \sgn(Z^n_s)\sigma_l^n (s) \, \di W^{l,n}_s
- \int_0^t \sum_{k=1}^{N-1} \sgn(Z^n_{t_k})\sigma_l^n (t_k)\bI_{[t_k,t_{k+1})}(s) \, \di W^{l,n}_s\Big|^2\\
&+3 \E\Big|\int_0^t\sum_{k=1}^{N-1} \sgn(Z^n_{t_k})\sigma_l^n (t_k)\bI_{[t_k,t_{k+1})}(s)\, \di W^{l,n}_s
- \int_0^t   \sum_{k=1}^{N-1} \sgn(Z_{t_k})\sigma_l (t_k)\bI_{[t_k,t_{k+1})}(s)  \, \di W^{l}_s\Big|^2\\
&+3 \E\Big|\int_0^t  \sum_{k=1}^{N-1} \sgn(Z_{t_k})\sigma_l (t_k)\bI_{[t_k,t_{k+1})}(s)\, \di W^{l}_s
- \int_0^t \sgn(Z_s)\sigma_l (s)\, \di W^{l}_s\Big|^2.
\ea
Applying the It\^o isometry to the first and the third stochastic integrals in \eqref{e:eps1} yields
\ba
\E\Big|\int_0^t & \sgn(Z^n_s)\sigma_l^n (s)\, \di W^{l,n}_s - \int_0^t \sgn(Z_s)\sigma_l (s)\, \di W^{l}_s\Big|^2\\
&\leq 3 \E\int_0^t \Big|\sgn(Z^n_s)\sigma_l^n (s)
-  \sum_{k=1}^{N-1} \sgn(Z^n_{t_k})\sigma_l^n (t_k)\bI_{[t_k,t_{k+1})}(s)\Big|^2
\, \di s\\
&+3 \E\Big|
\sum_{k=1}^{N-1} \sgn(Z^n_{t_k})\sigma_l^n (t_k)(W^{l,n}_{t_{k+1}}-W^{l,n}_{t_k})
-  \sum_{k=1}^{N-1} \sgn(Z_{t_k})\sigma_l (t_k) (W^{l}_{t_{k+1}}-W^{l}_{t_k}) \Big|^2\\
&+3 \E\int_0^t  \Big|\sum_{k=1}^{N-1} \sgn(Z_{t_k})\sigma_l (t_k)\bI_{[t_k,t_{k+1})}(s)
-  \sgn(Z_s)\sigma_l (s)\Big|^2\,\di s= E^{l,n,N}_1+E^{l,n,N}_2+E_3^{l,N}.
\ea
For $N$ large enough, the term $E_3^{l,N}$ is bounded by $3\e$ by \eqref{e:eps}.
By Lebesgue's theorem,
a.s.\ convergence $Z^n_s\to Z_s$ for each $s\in[0,\infty)$, and the fact that $Z(t_k)\neq 0$ a.s.\ for all $k=1,\dots,N$ we get
\ba
\lim_n E^{l,n,N}_1= E_3^{l,N}\leq 3\e.
\ea
Finally,
\ba
\limsup_n E^{l,n,N}_2=0,
\ea
because of the a.s.\ convergence $W^{l,n}_{t_k}\to W^l_{t_k}$ for all $k=1,\dots,N$,
a.s.\ convergence $Z^n_s\to Z_s$ and the fact that $Z(t_k)\neq 0$ a.s.\ for all $k=1,\dots,N$.

This finishes the proof of the Theorem.
\end{proof}

Now we finish the proof of Theorem \ref{t:existU} 1).
Let $\e\in (0,1]$ and $\delta\in(0,1]$ be fixed. For brevity, we omit these indices and denote
$(X,Y)=(X^{\delta},Y^{\delta})$ and $a_k:=a^\e_k$, $k\in\mathbb Z$.
Let us establish existence of a local solution is a neighbourhood of one membrane,
which is, for definiteness, located
at $x=0$. Formally this means that $a_0=0$, $\beta(t,0,Y)=\beta(t,Y)$,
$\theta(t,0,Y)=\theta(t,Y)$,
and all $\beta(\cdot, a_k,\cdot)$, $\theta(\cdot, a_k,\cdot)\equiv 0$ for $k\in \mathbb Z\backslash\{0\}$.

The equation \eqref{e:XY} takes the form
\ba
\label{e:XY0}
X_t&= x+\sum_{l=1}^m \int_0^t \sigma_{l}^0(s, X_s,Y_s)\,\di W^l_s
+ \int_0^t b^0(s,X_s,Y_s)\,\di s
+\delta \int_0^t \beta(s,Y_s)\,\di L^{0}_s (X),\\
Y^{i}_t&= y^i+\sum_{l=1}^m \int_0^t \sigma^{i}_l(s,X_s,Y_s)\,\di W^l_s
+\int_0^t b^i(s,X_s,Y_s)\,\di s
+ \delta\int_0^t  \theta^i(s,Y_s)\,\di L^0_s (X),\\
&\quad i=1,\dots,d,\ t\in[ 0,\infty).
\ea
Let for $n\in\mathbb N$
\ba
\phi_n(s)=\frac{k}{n},\quad s\in\Big[\frac{k}{n},\frac{k+1}{n}\Big),\quad k\in\mathbb N_0.
\ea
Now we consider a supplementary sequence of stochastic processes $(X^{n}, Y^{n})$, which are solutions of
SDEs with ``frozen'' coefficients on each time interval $[\frac{k}{n},\frac{k+1}{n})$:
\ba
\label{e:XY0n}
X_t^{n}&= x+\sum_{l=1}^m \int_0^t \sigma_{l}^0(\phi_n(s), X^n_{\phi_n(s)},Y^n_{\phi_n(s)})\,\di W^l_s
+ \int_0^t b^0(\phi_n(s),X^n_{\phi_n(s)},Y^n_{\phi_n(s)})\,\di s
+\delta \int_0^t \beta(\phi_n(s),Y^n_{\phi_n(s)})\,\di L^{0}_s (X^n),\\
Y^{i,n}_t&= y^i+\sum_{l=1}^m \int_0^t \sigma^{i}_l(\phi_n(s),X^n_{\phi_n(s)},Y^n_{\phi_n(s)})\,\di W^l_s
+\int_0^t b^i(\phi_n(s),X^n_{\phi_n(s)},Y^n_{\phi_n(s)})\,\di s
+ \delta\int_0^t  \theta^i(\phi_n(s),Y^n_{\phi_n(s)})\,\di L^0_s (X^n),\\
&\quad i=1,\dots,d,\ t\in[0,\infty).
\ea
On each time interval $[\frac{k}{n},\frac{k+1}{n})$, the coefficients of \eqref{e:XY0n} are time-constant
$\rF_{k/n}$-measurable random elements.
Therefore, on each time interval $[\frac{k}{n},\frac{k+1}{n})$, the solution $X^{n}$ exists as and unique,
provided,
\ba
\delta \|\beta\|_\infty<1,
\ea
see Theorem 2.3 in \cite{legall1984one}, and has a law of a time-changed skew Brownian motion with drift started at $X_{k/n}^{n}$.
On each time interval $[\frac{k}{n},\frac{k+1}{n})$,
the process $Y^{n}$ is obtained as a sum of a stochastic and Lebesgue--Stieltjes integrals
of $\rF_{k/n}$-measurable random constants.

For any $\delta$ small and fixed, Theorem \ref{t:Zncomp} implies that the sequence of vector valued processes
\ba
\label{e:seq}
\{(X^{n}, Y^{1,n},\dots, Y^{d,n}, W^1,\dots,W^m)\}_{n\in\mathbb N}
\ea
is weakly relatively compact in $C(\bR_+,\bR\times\bR^d\times\bR^m)$.

Let
\ba
\label{e:XXX}
(X, Y^{1},\dots, Y^{d}, W^1,\dots,W^m)
\ea
be a weak limit of some subsequence of \eqref{e:seq}. By the Skorokhod representation theorem, see Section 6
in \cite{billingsley2013convergence},
there is a probability space and copies of this subsequence
which converge to a copy of this limit locally uniformly with probability one. Therefore, by Theorem \ref{t:Znconv}
the process \eqref{e:XXX} is a weak solution to the SDE \eqref{e:XY0}.

Now, consider the original system \eqref{e:XY} with infinitely many membranes. We have to show that the solution
obtained from local solutions glued together does not blow up. We fix $\e\in(0,1]$ and $\delta\|\beta\|_\infty<1$ and
omit them in the notation.

Denote by $S(a_k)=(a_{k-1},a_{k+1})$ the stripe around the $x$-location of the $k$-th membrane.

A local solution $(X^{(0)},Y^{(0)})$ behaves like a solution of an SDE without local time terms until the
first hitting time
\ba
\tau_0&=\inf\Big\{t\in[0,\infty)\colon X^{(0)}_t \in \{a_k\}_{k\in\mathbb Z}\Big\}.
\ea
This solution exists as a weak solution driven by some $m$-dimensional Brownian motion $W=W^{(0)}$
on the interval $[0,\tau_0]$.

By the previous argument, we construct a process $(X^{(1)},Y^{(1)},W^{(1)})$
starting at $(X^{(0)}_{\tau_0},Y^{(0)}_{\tau_0})$ with some Brownian motion $W^{(1)}$ independent of the
behaviour of the process on $[0,\tau_0]$, until the hitting time
\ba
\tau_1&=\inf\{t\in[0,\infty)\colon X^{(1)}_t\notin S( X^{(0)}_{\tau_0})\}.
\ea
Define the Brownian motion $W$ on $[0,\tau_0+\tau_1]$ as
\ba
W_t=\begin{cases}
     W^{(0)}_t,\ t\in [0,\tau_0],\\
     W^{(0)}_{\tau_0} + W^{(1)}_{t-\tau_0},\ t\in (\tau_0,\tau_0+\tau_1],\\
    \end{cases}
\ea
and the process $(X,Y)$ as
\ba
(X_t,Y_t)=\begin{cases}
     (X^{(0)}_t,Y^{(0)}_t),\ t\in [0,\tau_0],\\
      (X^{(1)}_{t-\tau_0},Y^{(1)}_{t-\tau_0}),\ t\in (\tau_0,\tau_0+\tau_1].
    \end{cases}
\ea
The process $(X,Y,W)$ is a weak solution to \eqref{e:XY} on the interval  $[0,\tau_0+\tau_1]$.

Analogously, we construct the hitting times $\tau_2,\tau_3,\dots$ and
extend the solution $(X,Y,W)$ to each
random interval $[0,\tau_0+\cdots+\tau_n]$, $n\in\mathbb N_0$. It is left to show that
$\sum_{n=0}^\infty \tau_n=+\infty$ with probability one.

It follows from Lemma \ref{cor:estimate_hittings} that there are constants $c,C\in(0,\infty)$ that depend on
$\e$  and $\delta$ but do not depend on $n\in\mathbb N$
such that
\ba
\E (\tau_0+\cdots+\tau_n) &\geq c n,\\
\operatorname{Var} (\tau_0+\cdots+\tau_n) & \leq C n,\quad n\in\mathbb N.
\ea
Then, by Chebyshev's inequality
\ba
\P\Big(\tau_0+\cdots+\tau_n\leq \frac{cn}{2}\Big)&\leq \P\Big(|\tau_0+\cdots+\tau_n-\E (\tau_0+\cdots+\tau_n) \Big|\geq cn/2 \Big)\\
&\leq \frac{4}{c^2n^2}  \operatorname{Var} (\tau_0+\cdots+\tau_n)
\leq \frac{4C}{c^2n}\to 0,\quad n\to \infty.
\ea
Therefore, the solution $(X,Y,W)$ is well defined for all $t\in[0,\infty)$.

\subsection{Existence and uniqueness for time independent coefficients}

Let all functions $\beta$,   $b$, $\theta$, and $\sigma$ satisfy conditions of the second statement
of Theorem \ref{t:existU}, and let $\e$ and $\delta$ be fixed. 

Since existence of a solution $ (X^{\delta,\e}, Y^{\e,\delta})$ is already established,
it is sufficient to verify weak uniqueness in a sufficiently small neighbourhood of any point
$(x,y)\in \bR\times\bR^d$, see \S 6.6 in \cite{StrVar}. Note that although the exposition
in \cite{StrVar} focuses on diffusions without local times, the localization results there hold for
continuous stochastic processes in general.
The strong Markov property of the solution will follow from uniqueness, see \S 6.2 in \cite{StrVar}.

If a neighbourhood of $(x,y)$ does not contain points of the interface, the uniqueness follows
from the continuity and uniform ellipticity of the matrix $\sigma$, see Theorem 7.2.1 in \cite{StrVar}.

Therefore, we consider the system \eqref{e:XY} in a neighbourhood of one membrane $a_k^\e\times \bR^d$.
Without loss of generality we may assume that $k=0$ and
$a_0^\e=0$. For brevity, we omit the superscripts $\e$ and $\delta$.
Since outside the membrane
$(X,Y)$ is a diffusion with regular coefficients, we assume that the initial values are $x_0=0$ and $y_0\in\bR^d$.
With some abuse of notation we denote $\beta(y):=\beta(0,y)$, $\theta(y):=\theta(0,y)$
and consider the solution $(X,Y)$ of the following SDE with one membrane:
\ba
\label{e:XY0-}
X_t&= \sum_{l=1}^m \int_0^t \sigma_{l}^0(X_s,Y_s)\,\di W^l_s + \int_0^t b^0(X_s,Y_s)\,\di s
+\delta \int_0^t \beta(Y_s)\,\di L^0_s (X),\\
Y^{i}_t&= y^i+\sum_{l=1}^m \int_0^t \sigma^{i}_l(X_s,Y_s)\,\di W^l_s +\int_0^t b^i(X_s,Y_s)\,\di s
+ \delta\int_0^t   \theta^i(Y_s)\,\di L^0_s (X),\quad i=1,\dots,d.
\ea

We construct the transformation of $(X,Y)$ into a diffusion without the local time terms. Recall that
$\delta\in (0,  1\wedge \frac{1}{\|\beta\|_\infty})$. Let
\ba
\label{e:BB}
B(y):=\frac{1-\delta\beta(y)}{1+\delta\beta(y)},\quad y\in\bR^d.
\ea
Then,
\ba
0<\frac{1-\delta\|\beta\|_\infty}{1+\delta\|\beta\|_\infty}\leq B(y)
\leq \frac{1+\delta\|\beta\|_\infty}{1-\delta\|\beta\|_\infty}
<\infty.
\ea
Define the functions
\ba
F_-(x,y)&=x,\\
G_-(x,y) &= y+\delta x\theta(y),
\ea
and
\ba
F_+(x,y)&=x B( y),\\
 G_+(x,y) &= y-\delta x\theta(y).
\ea
It is clear that these functions map half-spaces into half-spaces as follows:
\ba
\label{e:37}
&(F_\pm,G_\pm)\colon (-\infty,0]\times\bR^d\to  (-\infty,0]\times\bR^d,\\
&(F_\pm,G_\pm)\colon [0,\infty)\times\bR^d\to  [0,\infty)\times\bR^d,\\
\ea
The Jacobi matrices $J_\pm$ of transformations $(x,y)\mapsto (F_\pm(x,y),G_\pm(x,y))$ equal to
\ba
J_-(x,y)=\begin{pmatrix}
          1 & 0\\
          \delta\theta(y) & \mathrm{Id}+ \delta x \nabla_y \theta(y)
         \end{pmatrix},\quad
J_+(x,y)=\begin{pmatrix}
          B(y) & x\nabla_y B(y)\\
          \delta\theta(y) & \mathrm{Id}- \delta x \nabla_y \theta(y)
         \end{pmatrix}.
\ea
In particular,
\ba
\label{e:jacob}
\det J_-(0,y)= 1\quad \text{and}\quad \det J_+(0,y)= B(y)\in (0,\infty),\quad y\in\bR^d.
\ea
Therefore, for any $y_0\in\bR^d$,
the functions
$(x,y)\mapsto (F_\pm(x,y),G_\pm(x,y))$ establish $C^2$-diffeomorphisms
 between a small neighbourhood  $\mathcal U(y_0)$ of the point $(0,y_0)$
 and its images $\mathcal V_\pm (y_0):=(F_\pm,G_\pm)(\mathcal U(y_0))$, respectively.

 We define the functions
\ba
F(x,y)&=F_-(x,y)\bI_{(-\infty,0)}(x)+F_+ (x,y)\bI_{[0,\infty)}(x),\\
G(x,y) &= G_-(x,y)\bI_{(-\infty,0)}(x)+G_+ (x,y)\bI_{[0,\infty)}(x).
\ea
The functions $F(x,y)$ and $G(x,y)$ are continuous in $\bR\times\bR^d$ and
are twice continuously differentiable on
$(-\infty,0)\times\bR^d $ and $(0,\infty)\times\bR^d $.

It follows from \eqref{e:37} that the function $(F,G)$
is a homeomorphism between $\mathcal U(y_0)$ and
\ba
\mathcal V(y_0):=
\Big(\mathcal V_- (y_0)\cap (-\infty,0]\times\bR^d\Big) \bigcup \Big(\mathcal V_+ (y_0)\cap  [0,\infty)\times\bR^d\Big).
\ea
Let
\ba
(\Phi,\Psi):=(F,G)^{-1}\colon \mathcal V(y_0)\to \mathcal U(y_0)
\ea
be the
inverse mapping.

For a solution $(X,Y)$ starting at $(0,y_0)$ we define the first exit time
\ba
\tau:=\inf\{t\in [0,\infty)\colon (X_t,Y_t)\notin \mathcal U(y_0)\}
\ea
and the processes
\ba
U_t&:=F (X_t, Y_t), \\
V_t&:=G (X_t, Y_t),\quad t\in [0,\tau].
\ea
Equivalently,
\ba
\label{e:PP}
X_t&=\Phi(U_t,V_t),\\
Y_t&=\Psi(U_t,V_t),\quad t\in [0,\tau].
\ea
Moreover,
\ba
\tau:=\inf\{t\in [0,\infty)\colon (U_t,V_t)\notin \mathcal V(y_0)\}.
\ea
\begin{prp}
The process $(X,Y)_{t\in[0, \tau]}$ is a (weak) solution of \eqref{e:XY0} if and only if
$(U,V)_{t\in [0,\tau]}$ is a (weak) solution of
\ba
\label{e:UV0}
U_t
&= \int_0^t   \Big[ b^0(\cdot,\cdot)+ \phi^{0}(\cdot,\cdot) \Big]\circ \Big(\Phi(U_s,V_s),\Psi(U_s,V_s)\Big) \,\di s\\
&+\sum_{l=1}^m\int_0^t   \Big[ \sigma_{l}^0 (\cdot,\cdot) + \phi_{l}^{0}(\cdot,\cdot)     \Big]
\circ \Big(\Phi(U_s,V_s),\Psi(U_s,V_s)\Big)  \,\di W^l_s,   \\
V^{i}_t&=y^i
+\int_0^t \Big[ b^i(\cdot,\cdot)
-\psi^{i}(\cdot,\cdot)\Big] \circ\Big(\Phi(U_s,V_s),\Psi(U_s,V_s)\Big)  \,\di s\\
&+\sum_{l=1}^m \int_0^t \Big[  \sigma_{l}^i(\cdot,\cdot)
 -\psi_{l}^{i}(\cdot,\cdot)  \Big] \circ \Big(\Phi(U_s,V_s),\Psi(U_s,V_s)\Big)   \,\di W^l_s, \quad i=1,\dots,d,\\
\ea
where
\ba
\phi^{0}(x,y)&=
(B(y)-1)  b^0(x,y) \bI_{(0,\infty)}
+ \sum_{i=1}^d B_{y^i}(y) \Big(  x^+b^i(x,y)+   \Sigma^{0i}(x,y)\bI_{(0,\infty)}(x)\Big)
\\+ & \frac12 \sum_{i,j=1}^d x^+B_{y^iy^j}(y)\Sigma^{ij}(x,y),\\
\phi_{l}^{0}(x,y)
&=(B(y)-1)\bI_{(0,\infty)}  \sigma_{l}^0(x,y) + x^+ \sum_{i=1}^d B_{y^i}(y) \sigma_{l}^i(x,y),\\
\psi^{i}(x,y)
&= \delta\Big( \sum_{j=1}^d \theta_{y^j}^i(y)\big(|x|b^j(x,y)+\Sigma^{0j}(x,y)\big)+\frac{|x|}{2}\sum_{j,k=1}^d \theta^i_{y^jy^k}(y)\Sigma^{jk}(x,y)+\theta^i(y)b^0(x,y)\sgn(x)\Big),\\
\psi_{l}^{i}(x,y)&=\delta\Big(|x|\sum_{j=1}^d \theta^i_{y^j}(y)\sigma_{l}^j(x,y)+\theta^i(y)\sigma_l^0(x,y)\sgn(x)\Big),\quad i=1,\dots,d,\ l=1,\dots,m.
\ea
\end{prp}
\begin{proof}
Recall the Tanaka formulas:
\ba
X_t^+&=\int_0^t \Big(\bI(X_s>0)+ \frac12 \bI(X_s=0)\Big) \,\di X_s + \frac12 L^0_t(X),\\
|X_t|&=\int_0^t \sgn X_s \,\di X^\delta_s + L^0_t(X).
\ea
Hence the application of
 the It\^o formula to $B(Y)-1$,   $\theta(Y)$ and the product It\^o formula yields  that
 \ba
\label{e:UVXY}
U_t
&= \int_0^t   \Big(  b^0+ \phi^{0}\Big)(X_s,Y_s) \,\di s
+\sum_{l=1}^m\int_0^t   \Big( \sigma_{l}^0  + \phi_{l}^{0}    \Big)(X_s,Y_s)\,\di W^l_s,   \\
V^{i}_t&=y^i
+\int_0^t \Big( b^i -\psi^{i}\Big)(X_s,Y_s) \,\di s
+\sum_{l=1}^m \int_0^t \Big(  \sigma_{l}^i -\psi_{l}^{i}  \Big)(X_s,Y_s)\,\di W^l_s, \quad i=1,\dots,d,\\
\ea
and the representation \eqref{e:UV0} follows immediately.
To transform the system \eqref{e:UV0} into \eqref{e:XY0}, we apply the It\^o formula with local times as proven by Peskir
\cite{peskir2007change}.
\end{proof}


Since $(F_\pm$, $G_\pm)$ have non-degenerate Jacobians in the neighbourhood $\mathcal U(y_0)$ of the initial point, see \eqref{e:jacob},
Assumption \textbf{A}$_{\Sigma}^\text{elliptic}$ implies the uniform ellipticity of the
diffusion matrix of $(U,V)$ in the neighbourhood $\mathcal V(y_0)$.
Therefore, uniqueness of $(U,V)$ on $[0,\tau]$ follows from Gao \cite[p.\ 139]{gao1993martingale}.
Since $(F,G)$ is a bijection, we obtain uniqueness of $(X,Y)$ on $[0,\tau]$, what finishes the proof of the Theorem.

\section{Dynamics with many membranes. Limit theorem for scaling of  local times sums\label{s:many}}

The main result of this section is Theorem \ref{thm:conv_CAF} that provides, in particular, a
functional limit theorem for sums of local times if the processes $\{X^{\e,\delta}\}$
in the limit as the distance between the membranes converges to 0.

\begin{thm}
\label{thm:conv_CAF}
Assume that conditions of Theorem \ref{t:A} are satisfied, and assume that for each $\e,\delta\in(0,1]$
$(X^{\e, \delta }, Y^{\e, \delta }, W^{\e, \delta },(\rF_t^{\e,\delta}))$ is a weak solution of \eqref{e:XY}.
Assume that
\ba
(X^{\e, \delta }_t, Y^{\e, \delta }_t, W^{\e, \delta }_t)_{t\in[0,\infty)}\Rightarrow (X_t,Y_t,W_t)_{t\in[0,\infty)},\quad \e,\delta \to 0,
\ea
in the space of continuous functions.
Let $\{f^{\e,  \delta},g^{\e,  \delta},h^{\e,  \delta} \}_{\e,\delta\in(0,1]}$ be a
family of bounded measurable vector- or matrix-valued functions
defined on $\bR_+\times\bR\times \bR^d$, such that they
locally uniformly converge to continuous functions $f, g,h$, respectively, as $\e,\delta\to 0$.
Then, the sequence
\ba
\Big(X^{\e,\delta}_t,
Y^{\e,\delta}_t,
\int_0^t f^{\e,  \delta}(s,X_s^{\e,\delta},Y_s^{\e,  \delta })\,\di s,
\int_0^t g^{\e,  \delta}(s,X_s^{\e,\delta},Y_s^{\e,  \delta })\,\di W^{\e,  \delta }_s,
\e \sum_{k=-\infty}^\infty \int_0^t h^{\e,  \delta}(s,a_k^\e,Y_s^{\e,  \delta })\,\di L^{a_k^\e}_s (X^{\e,\delta})
\Big)_{t\in[0,\infty)}
\ea
weakly converges in the space of continuous functions to
\ba
\Big(X_t,
Y_t,
\int_0^t f(s,X_s,Y_s)\,\di s,
\int_0^t g(s,X_s,Y_s)\,\di W_s,
\int_0^t h(s,X_s,Y_s) \frac{\Sigma^{00}(s,X_s, Y_s)}{d(X_s)}\,\di s
\Big)_{t\in[0,\infty)}
\ea
as $\e,\delta\to 0$.
In particular,
\ba
\Big(
\e \sum_{k=-\infty}^\infty L^{a_k^{\e }}_t (X^{\e, \delta })\Big)_{t\in[0,\infty)}
\Rightarrow \Big(\int_0^t \frac{\Sigma^{00}(s,X_s, Y_s)}{d(X_s)}\, \di s\Big)_{t\in[0,\infty)},\quad \e,\delta \to 0,
\ea
in $C(\bR_+,\bR)$.
\end{thm}

The proof of Theorem \ref{thm:conv_CAF} will be presented at the end of this section.
To prepare for it, we establish several auxiliary results showing that the asymptotics
of sums of local times coincide with the asymptotics of the sums of their conditional expectations.

The process $(X^{\e,\delta}, Y^{\e,\delta},W^{\e,\delta},(\rF_t^{\e,\delta}))$
given, we define the sequence of stopping times $(\tau^{\e,\delta}_j)_{j\in \mathbb N_0}$ as follows:
\ba
\tau^{\e,\delta}_0&=\inf\Big\{t\geq 0\colon X^{\e,\delta}_t\in \{a_k^\e\}_{k\in\mathbb Z}\Big\},\\
\tau^{\e,\delta}_{j+1}&=\inf\Big\{t>\tau^{\e,\delta}_{j} \colon X^{\e,\delta}_t\in \{a_k^\e\}_{k\in\mathbb Z}
\backslash \{ X^{\e,\delta}_{\tau^{\e,\delta}_j}\}
\Big\},\quad j\in \mathbb N_0.
\ea
Moreover, we define a sequence of random indices $k^{\e,\delta}=(k_j^{\e,\delta})_{j\in\mathbb N_0}$ such that
\ba
a^\e_{k_j^{\e,\delta}}=X^{\e,\delta}_{\tau^{\e,\delta}_{j }},\quad j\in\mathbb N_0.
\ea
The process $X^{\e,\delta}$ subsequently visits the set of membranes $\{a_k^\e\}_{k\in\mathbb Z}$, and
$\tau^{\e,\delta}_j$ is the time instant of the $j$-th changeover. Furthermore, for all $j\in\mathbb N_0$
\ba
|k_{j+1}^{\e,\delta}-k_{j}^{\e,\delta}|=1.
\ea

By estimate \eqref{e:tau1} in Lemma \ref{cor:estimate_hittings}, the stopping times $\tau^{\e,\delta}_j$,
$j\in \mathbb N_0$, are finite
and integrable.

\begin{lem}
\label{l:loctime}
Let Assumptions \emph{\textbf{A}}$_\mathrm{coeff}^{C_b}$, \emph{\textbf{A}}$^\mathrm{sep}_a$, \emph{\textbf{A}}$_{\Sigma^{00}}$ hold.
Let for any $\e,\delta\in(0,1]$,
$(\xi^{\e,\delta}_j)_{j\in\mathbb N_0}$ be a sequence of random variables
adapted to the filtration  $(\rF^{\e,\delta}_{\tau^{\e,\delta}_j})_{j\in\mathbb N_0}$.
Assume that there is $C\in(0,\infty)$ such that
$\P(|\xi^{\e,\delta}_j|\leq C)=1$ for all $j\in\mathbb N_0$ and all $\e,\delta\in (0,1]$.
Then, for any $N\in(0,\infty)$
\ba
\label{e:sumL}
&\max_{0\leq n\leq {N\e^{-2}} }
 \e \Bigg| \sum_{j=0}^n \sum_{k=-\infty}^\infty\xi^{\e,\delta}_j
 \Big(L^{a_k^\e}_{\tau^{\e,\delta}_{j+1}}  (X^{\e,\delta})
 -L^{a^\e_k}_{\tau^{\e,\delta}_{j }} (X^{\e,\delta})
 -\E\Big[L^{a_k^\e}_{\tau^{\e,\delta}_{j+1}}  (X^{\e,\delta})
 -L^{a_k^\e}_{\tau^{\e,\delta}_{j }}  (X^{\e,\delta}) \Big| \rF^{\e,\delta}_{\tau^{\e,\delta}_{j }} \Big] \Big)\Bigg|
        \overset{\P}{\to} 0,\  \e\to0,
\ea
and
\ba
\label{e:sumtau}
&\max_{0\leq n\leq { N\e^{-2}} }
 \sum_{j=0}^n \Big(  {\tau^{\e,\delta}_{j+1}} - {\tau^{\e,\delta}_{j }}
 -\E\Big[{\tau^{\e,\delta}_{j+1}} - {\tau^{\e,\delta}_{j }} \Big| \rF^{\e,\delta}_{\tau^{\e,\delta}_{j }} \Big]\Big)
        \overset{\P}{\to} 0, \quad \e\to 0.
\ea
\end{lem}
\begin{proof}
First we prove the limit \eqref{e:sumL}.
Note that
\ba
\label{e:diffL}
\sum_{k=-\infty}^\infty
 \Big(L^{a_k^\e}_{\tau^{\e,\delta}_{j+1}}  (X^{\e,\delta})
 -L^{a^\e_k}_{\tau^{\e,\delta}_{j }} (X^{\e,\delta})\Big)
 = L^{a^\e_{k_j^{\e,\delta}} }_{\tau^{\e,\delta}_{j+1}}  (X^{\e,\delta})
 -L^{a^\e_{k_j^{\e,\delta}} }_{\tau^{\e,\delta}_{j }} (X^{\e,\delta}).
\ea
The integrability of the right-hand-side in \eqref{e:diffL} follows from Lemma \ref{cor:estimate_hittings}.
The sequence
\ba
 \left(\e  \sum_{j=0}^n \xi_j^{\e,\delta}\Big(L^{a^\e_{k_j^{\e,\delta}}}_{\tau^{\e,\delta}_{j+1}}  (X^{\e,\delta})
 -L^{a^\e_{k_j^{\e,\delta}}}_{\tau^{\e,\delta}_{j }} (X^{\e,\delta})
 -\E\Big[L^{a^\e_{k_j^{\e,\delta}}}_{\tau^{\e,\delta}_{j+1}}  (X^{\e,\delta})
 -L^{a^\e_{k_j^{\e,\delta}}}_{\tau^{\e,\delta}_{j }}  (X^{\e,\delta})\Big| \rF^{\e,\delta}_{\tau^{\e,\delta}_{j }} \Big]
 \Big)\right)_{n\in\mathbb N_0}
\ea
is a martingale. So, to prove the Lemma it suffices to verify that
\ba
\lim_{\e\to 0}\e^2
\sum_{0\leq j\leq N\e^{-2}}
\E \Big[(\xi^{\e,\delta}_j)^2\Big(L^{a^\e_{k_j^{\e,\delta}}}_{\tau^{\e,\delta}_{j+1}}  (X^{\e,\delta})
-L^{a^\e_{k_j^{\e,\delta}}}_{\tau^{\e,\delta}_{j }} (X^{\e,\delta})\Big)^2\Big]=0.
\ea
This estimate follows from the estimate \eqref{e:L1} of Lemma \ref{cor:estimate_hittings}.

The proof of \eqref{e:sumtau} follows analogously from the
estimate \eqref{e:tau1} of Lemma \ref{cor:estimate_hittings}.
\end{proof}

\begin{lem}
\label{l:Eloctime}
Let assumptions \emph{\textbf{A}}$_\mathrm{coeff}^{C_b}$, \emph{\textbf{A}}$^\mathrm{sep}_a$, \emph{\textbf{A}}$_{\Sigma^{00}}$ hold.
Then, for any $N\in(0,\infty)$, the following limits hold as $\e,\delta\to 0$:
\begin{align}
\label{e:LF}
& \max_{0\leq j\leq {N\e^{-2}} } \e \E\Big[L^{a^\e_{k_j^{\e,\delta}}}_{\tau^{\e,\delta}_{j+1}}  (X^{\e,\delta})
 -L^{a^\e_{k_j^{\e,\delta}}}_{\tau^{\e,\delta}_{j }}  (X^{\e,\delta}) \Big|\rF^{\e,\delta}_{\tau^{\e,\delta}_{j }} \Big]
        \overset{\P}{\to} 0,\\
\label{e:L}
&\max_{0\leq j\leq {N\e^{-2}} }
 \e \Big( L^{a^\e_{k_j^{\e,\delta}}}_{\tau^{\e,\delta}_{j+1}}  (X^{\e,\delta})
 -L^{a^\e_{k_j^{\e,\delta}}}_{\tau^{\e,\delta}_{j }} (X^{\e,\delta}) \Big) \overset{\P}{\to} 0,\\
\label{e:tauF}
&\max_{0\leq j\leq N\e^{-2}}
\E \Big[ \tau^{\e,\delta}_{j+1}-\tau^{\e,\delta}_{j}\Big| \rF^{\e,\delta}_{\tau^{\e,\delta}_{j }} \Big] \overset{\P}{\to} 0,\\
\label{e:tau}
&  \max_{0\leq j\leq N\e^{-2}} ( \tau^{\e,\delta}_{j+1}-\tau^{\e,\delta}_{j}) \overset{\P}{\to} 0.
\end{align}

\end{lem}
\begin{proof}
The limits \eqref{e:LF} and \eqref{e:tauF} follow immediately from the estimates
\eqref{e:L1} and \eqref{e:tau1} of Lemma \ref{cor:estimate_hittings}.

From \eqref{e:sumL} in Lemma \ref{l:loctime} it follows that
\ba
\label{e:maxL}
& \max_{0\leq j\leq {N\e^{-2}} }
 \e \Bigg|
 L^{a^\e_{n_j^{\e,\delta}}}_{\tau^{\e,\delta}_{j+1}}  (X^{\e,\delta})
 -L^{a^\e_{n_j^{\e,\delta}}}_{\tau^{\e,\delta}_{j }}  (X^{\e,\delta})
 -\E\Big[L^{a^\e_{n_j^{\e,\delta}}}_{\tau^{\e,\delta}_{j+1}}  (X^{\e,\delta})
 -L^{a^\e_{n_j^{\e,\delta}}}_{\tau^{\e,\delta}_{j }}  (X^{\e,\delta})\Big| \rF^{\e,\delta}_{\tau^{\e,\delta}_{j }} \Big]
 \Bigg|\\
&=\max_{0\leq j\leq {N\e^{-2}} }
 \e \Bigg| \sum_{k=-\infty}^\infty
 \Big(L^{a_k^\e}_{\tau^{\e,\delta}_{j+1}}  (X^{\e,\delta})
 -L^{a^\e_k}_{\tau^{\e,\delta}_{j }} (X^{\e,\delta})
 -\E\Big[L^{a_k^\e}_{\tau^{\e,\delta}_{j+1}}  (X^{\e,\delta})
 -L^{a_k^\e}_{\tau^{\e,\delta}_{j }}  (X^{\e,\delta}) \Big| \rF^{\e,\delta}_{\tau^{\e,\delta}_{j }} \Big] \Big)\Bigg|\\
&        \overset{\P}{\to} 0,\  \e\to0.
\ea
Then, \eqref{e:L} follows from \eqref{e:LF} and \eqref{e:maxL},

The estimate \eqref{e:tau} is obtained analogously.
\end{proof}

The following Lemma is reformulation of Corollary \ref{cor:accurate_estim_X} for the solution $(X^{\e,\delta}, Y^{\e,\delta})$.

\begin{lem}
\label{l:estim_X}
Under the assumptions
$\mathbf{A}_\mathrm{coeff}^{C_b}$,
$\mathbf{A}_\mathrm{coeff}^\mathrm{Lip}$,
\emph{\textbf{A}}$^\mathrm{sep}_a$,
\emph{\textbf{A}}$_{\Sigma^{00}}$,
the following estimates hold as $\e,\delta\to 0$:
\begin{align}
\label{e:accL}
\E\Big[L^{ a^\e_{k_j^{\e,\delta}} }_{\tau^{\e,\delta}_{j+1}}
(X^{\e,\delta}) -L^{a^\e_{k_j^{\e,\delta}} }_{\tau^{\e,\delta}_{j }}  (X^{\e,\delta}) \Big| \rF^{\e,\delta}_{\tau^{\e,\delta}_{j }} \Big]
&  =  \frac{ 2  (a^\e_{k_j^{\e,\delta}} -a^\e_{k_j^{\e,\delta}-1})   (a^\e_{k_j^{\e,\delta}+1}-a^\e_{k_j^{\e,\delta}} )}
{ a^\e_{k_j^{\e,\delta}+1} - a^\e_{k_j^{\e,\delta}-1} }  +\mathcal O(\e^2)+ \mathcal O(\e\delta),\\
\label{e:acctau}
\E\Big[\tau^{\e,\delta}_{j+1} - \tau^{\e,\delta}_j \Big|\rF^{\e,\delta}_{\tau^{\e,\delta}_j} \Big]
&= \frac{(a^\e_{k_j^{\e,\delta}} -a^\e_{k_j^{\e,\delta}-1})   (a^\e_{k_j^{\e,\delta}+1}-a^\e_{n_j^{\e,\delta}} )}
{\Sigma^{00}( \tau^{\e,\delta}_{j }  ,X^{\e,\delta}_{\tau^{\e,\delta}_{j } }, Y^{\e,\delta}_{\tau^{\e,\delta}_{j } } )}
+\mathcal O(\e^3)+\mathcal O(\e^2\delta)\\
&= \frac{(a^\e_{k_j^{\e,\delta}} -a^\e_{k_j^{\e,\delta}-1})   (a^\e_{k_j^{\e,\delta}+1}-a^\e_{k_j^{\e,\delta}} )}
{\Sigma^{00}( \tau^{\e,\delta}_{j }  ,  a^\e_{k_j^{\e,\delta}}  , Y^{\e,\delta}_{\tau^{\e,\delta}_{j } } )}
+\mathcal O(\e^3)+\mathcal O(\e^2\delta) ,\\
\label{e:accX}
\E\Big[X^{\e,\delta}_{\tau^{\e,\delta}_{j+1}} - X^{\e,\delta}_{\tau^{\e,\delta}_j} \Big|\rF^{\e,\delta}_{\tau^{\e,\delta}_j} \Big]
& =
 \frac{b^0( \tau^{\e,\delta}_{j }  ,a^\e_{k_j^{\e,\delta}} , Y^{\e,\delta}_{\tau^{\e,\delta}_{j } } )}{\Sigma^{00}( \tau^{\e,\delta}_{j }  ,a^\e_{k_j^{\e,\delta}} , Y^{\e,\delta}_{\tau^{\e,\delta}_{j } } )}
 (a^\e_{k_j^{\e,\delta}} -a^\e_{k_j^{\e,\delta}-1})   (a^\e_{k_j^{\e,\delta}+1}-a^\e_{k_j^{\e,\delta}} ) \\
& +\beta( \tau^{\e,\delta}_{j }  ,a^\e_{k_j^{\e,\delta}} , Y^{\e,\delta}_{\tau^{\e,\delta}_{j } } )\frac{2(a^\e_{k_j^{\e,\delta}} -a^\e_{k_j^{\e,\delta}-1})   (a^\e_{k_j^{\e,\delta}+1}-a^\e_{k_j^{\e,\delta}} )}
{a^\e_{k_j^{\e,\delta}+1} - a^\e_{k_j^{\e,\delta}-1} } \delta
+\mathcal O(\e^2)+\mathcal O(\e\delta),
\end{align}
where all error terms $\mathcal O=\mathcal O_{\omega,j,\e,\delta}$ are such that
\ba
\operatorname*{ess\,sup}_{\omega}\sup_{j,\e,\delta}
\Big( \frac{|\mathcal O(\e^2)|}{\e^2}
+ \frac{|\mathcal O(\e^3)|}{\e^3}
+ \frac{|\mathcal O(\e\delta)|}{\e\delta}
+ \frac{|\mathcal O(\e^2\delta)|}{\e^2\delta} \Big)<\infty.
\ea
\end{lem}

\begin{lem}
\label{l:55}
Under the assumptions
$\mathbf{A}_\mathrm{coeff}^{C_b}$,
$\mathbf{A}_\mathrm{coeff}^\mathrm{Lip}$,
$\mathbf{A}^\mathrm{sep}_a$,
$\mathbf{A}_{\Sigma^{00}}$,
the following ucp-limits hold as $\e,\delta\to 0$:
\ba
\label{e:csumt}
\sum_{\tau^{\e,\delta}_{j}\leq t}
\E\Big[{\tau^{\e,\delta}_{j+1}} - {\tau^{\e,\delta}_{j }} \Big|\rF^{\e,\delta}_{\tau^{\e,\delta}_{j }} \Big]
\to  t,
\ea
\ba
\label{e:csumt1}
\sum_{\tau^{\e,\delta}_{j}\leq t}
\frac{(a^\e_{k_j^{\e,\delta}}-a^\e_{k_j^{\e,\delta}-1})(a^\e_{k_j^{\e,\delta}+1}-a^\e_{k_j^{\e,\delta}})}
{\Sigma^{00}( \tau^{\e,\delta}_{j }  ,a^\e_{k_j^{\e,\delta}}, Y^{\e,\delta}_{\tau^{\e,\delta}_{j } } )}
              \to t,
\ea
\ba
\label{e:csumL}
\sum_{\tau^{\e,\delta}_{j}\leq t}   \frac{a^\e_{k_j^{\e,\delta}+1}-a^\e_{k_j^{\e,\delta}-1} }
{2 {\Sigma^{00}(\tau^{\e,\delta}_{j }  ,a^\e_{k_j^{\e,\delta}}, Y^{\e,\delta}_{\tau^{\e,\delta}_{j } } )}  }
\E\Big[ L^{ a^\e_{k_j^{\e,\delta}} }_{\tau^{\e,\delta}_{j+1}}  (X^{\e,\delta}) -L^{a^\e_{k_j^{\e,\delta}} }_{\tau^{\e,\delta}_{j }}  (X^{\e,\delta})\Big| \rF^{\e,\delta}_{\tau^{\e,\delta}_{j }} \Big]
\to t,
\ea
\ba
\label{e:csumL1}
\sum_{\tau^{\e,\delta}_{j}\leq t}   \frac{ a^\e_{k_j^{\e,\delta}+1}-a^\e_{k_j^{\e,\delta}-1} }{2 {\Sigma^{00}(\tau^{\e,\delta}_{j }  ,
 a^\e_{k_j^{\e,\delta}}, Y^{\e,\delta}_{\tau^{\e,\delta}_{j } } )}  }  \Big(L^{ a^\e_{k_j^{\e,\delta}} }_{\tau^{\e,\delta}_{j+1}}  (X^{\e,\delta}) -L^{a^\e_{k_j^{\e,\delta}} }_{\tau^{\e,\delta}_{j }}  (X^{\e,\delta})  \Big)
\to t.
\ea
\end{lem}
\begin{proof}
To show  \eqref{e:csumt}, we write
for each $T\in[0,\infty)$:
\ba
\label{e:sup1}
&\sup_{t\in[0,T]}\Big| \sum_{\tau^{\e,\delta}_{j}\leq t}\E\Big[{\tau^{\e,\delta}_{j+1}}
- {\tau^{\e,\delta}_{j }} \Big|\rF_{\tau^{\e,\delta}_{j }} \Big] -t \Big|\\
&\leq
\sup_{t\in[0,T]} \Big|
\sum_{\tau^{\e,\delta}_{j}\leq t}(\tau^{\e,\delta}_{j+1}
- {\tau^{\e,\delta}_{j }} ) -t \Big|
+
\sup_{t\in[0,T]}\Big| \sum_{\tau^{\e,\delta}_{j}\leq t}  \Big( \tau^{\e,\delta}_{j+1}
- {\tau^{\e,\delta}_{j }}
- \E\Big[{\tau^{\e,\delta}_{j+1}} - {\tau^{\e,\delta}_{j }}\Big|\rF_{\tau^{\e,\delta}_{j }} \Big] \Big) \Big|\\
&\leq
\tau^{\e,\delta}_0
+ \max_{\tau^{\e,\delta}_{j}\leq T} \Big|\tau^{\e,\delta}_{j+1} -\tau^{\e,\delta}_{j} \Big|
+\sup_{t\in[0,T]}\Big| \sum_{\tau^{\e,\delta}_{j}\leq t}
\Big(\tau^{\e,\delta}_{j+1} - {\tau^{\e,\delta}_{j }}    - \E\Big[{\tau^{\e,\delta}_{j+1}}
- {\tau^{\e,\delta}_{j }} \Big|\rF_{\tau^{\e,\delta}_{j }} \Big] \Big)\Big|.
\ea
 Let $\nu^{\e,\delta}= (\nu^{\e,\delta}_t)$ be a counting
process of visits to $\{a_k^\e\}_{k\in\mathbb Z}$,
\ba
 \nu^{\e,\delta}_t = k \quad \Leftrightarrow\quad
\tau_0^{\e,\delta} < \cdots < \tau_k^{\e,\delta}  \leq t,\ \tau_{k+1}^{\e,\delta}>t  .
\ea
As in Lemma 4.1 of \cite{Aryasova+24} we show that for each $T\in[0,\infty)$ there is $N\in(0,\infty)$ such that
\ba
\label{e:nuN}
\lim_{\e\to 0}\P( \nu^{\e,\delta}_T >NT\e^{-2} )=0.
\ea
The proof of \eqref{e:csumt} follows from \eqref{e:nuN}, \eqref{e:sup1} and Lemmas \ref{l:loctime} and \ref{l:Eloctime}.
The proof of \eqref{e:csumt1} follows from \eqref{e:csumt}, \eqref{e:acctau} and \eqref{e:nuN}.
The proof of \eqref{e:csumL} follows from \eqref{e:csumt1} and \eqref{e:accL} and \eqref{e:acctau}.
The proof of \eqref{e:csumL1} follows from \eqref{e:csumL} and \eqref{e:sumL}.
\end{proof}

\begin{lem}
\label{cor:conv_CAF_to_t}
Under the assumptions
$\mathbf{A}_\mathrm{coeff}^{C_b}$,
$\mathbf{A}_\mathrm{coeff}^\mathrm{Lip}$,
$\mathbf{A}_d^{\mathrm{Lip}_b}$,
$\mathbf{A}_{\Sigma^{00}}$,
the following ucp-limits hold as $\e,\delta\to 0$:
\ba
\label{e:csumt1d}
    \e^2  \sum_{\tau^{\e,\delta}_{j}\leq t}  
     \frac{d^2 ( a^\e_{k_j^{\e,\delta}})}{\Sigma^{00}(\tau^{\e,\delta}_j,a^\e_{k_j^{\e,\delta}}, Y^{\e,\delta}_{\tau^{\e,\delta}_{j } } )}
\to t,
\ea
\ba
\label{e:csumL1d}
\e   \sum_{\tau^{\e,\delta}_j \leq t}
     \frac{d  (a^\e_{k_j^{\e,\delta}})}{\Sigma^{00}(\tau^{\e,\delta}_j,a^\e_{k_j^{\e,\delta}}, Y^{\e,\delta}_{\tau^{\e,\delta}_j } )}
\Big(
L^{a^\e_{k_j^{\e,\delta}}}_{\tau^{\e,\delta}_{j+1}}  (X^{\e,\delta})
-L^{a^\e_{k_j^{\e,\delta}}}_{\tau^{\e,\delta}_{j }} (X^{\e,\delta})\Big)
\to t.
\ea
\end{lem}
\begin{proof}
Due to the uniform continuity of $d$,
uniformly over $j$, $\e$ etc.\
we have
\ba
a^\e_{n_j^{\e,\delta}}-a^\e_{n_j^{\e,\delta}-1}&=\e d(a^\e_{n_j^{\e,\delta}})
+
\e (d(a^\e_{n_j^{\e,\delta}-1}) - d(a^\e_{n_j^{\e,\delta}})) + o(\e)\\
&=
\e d(a^\e_{n_j^{\e,\delta}}) +  o(\e),
\\
a^\e_{n_j^{\e,\delta}+1}-a^\e_{n_j^{\e,\delta}}&= \e d(a^\e_{n_j^{\e,\delta}}) +  o(\e).
\ea
The limit \eqref{e:csumt1d} follows from \eqref{e:csumt1}, and the limit \eqref{e:csumL1d} follows from \eqref{e:csumL1}.
\end{proof}

\begin{lem}
\label{e:tot}
Suppose that assumptions
$\mathbf{A}_\mathrm{coeff}^{C_b}$,
$\mathbf{A}_\mathrm{coeff}^\mathrm{Lip}$,
$\mathbf{A}_d^{\mathrm{Lip}_b}$,
$\mathbf{A}_{\Sigma^{00}}$ hold.
Assume that the family $\{X^{\e,\delta},Y^{\e,\delta}\}_{\e,\delta\in(0,1]}$ is weakly relatively compact.
Let $\{\e_n\}$, $\{\delta_n\}\subseteq(0,1]$ be such that $\lim \e_n=\lim \delta_n=0$.
Then
\ba
\e_n
\sum_{k=-\infty}^\infty\int_0^t \frac{d(X^{\e_n,\delta_n}_s)}{\Sigma^{00}(s,X^{\e_n,\delta_n}_s, Y^{\e_n,\delta_n}_s)}\,
\di L^{a_k^\e}_s (X^{\e_n,\delta_n})\to t
\ea
in u.c.p.
\end{lem}
\begin{proof}
We introduce a random function
\ba
H_n(t):=
\frac{d  ( x )}{\Sigma^{00}(0,x,y )}
\mathbb I_{[0, \tau^{\e_n,\delta_n}_0)}(t)+
\sum_{j=0}^\infty
     \frac{d  (a^{\e_n}_{k_j^{\e_n,\delta_n}})}{\Sigma^{00}(\tau^{\e_n,\delta_n}_j,a^{\e_n}_{k_j^{\e_n,\delta_n}},
     Y^{\e_n,\delta_n}_{\tau^{\e_n,\delta_n}_j } )}
\mathbb I_{[\tau^{\e_n,\delta_n}_{j}, \tau^{\e_n,\delta_n}_{j+1} )}(t).
\ea
In follows from \eqref{e:tau} and \eqref{e:nuN} and the weak relative compactness of the family
$\{X^{\e_n,\delta_n},Y^{\e_n,\delta_n}\}_{n\in\mathbb N}$
that
\ba
\label{e:h0}
\Big|H_n(t)
-
\frac{d(X^{\e_n,\delta_n}_t)}{\Sigma^{00}(t,X^{\e_n,\delta_n}_t, Y^{\e_n,\delta_n}_t)}
\Big|\to 0
\ea
in u.c.p.
We have
\ba
\Big|     & \e_n    \sum_{k=-\infty}^\infty \int_0^t H_n(s)\,\di L^{a_k^{\e_n}}_s (X^{\e_n,\delta_n})
-
\e_n   \sum_{\tau^{\e_n,\delta_n}_j \leq t}
\frac{d  (a^{\e_n}_{k_j^{\e_n,\delta_n}})}{\Sigma^{00}(\tau^{\e_n,\delta_n}_j,a^\e_{k_j^{\e_n,\delta_n}},
Y^{\e_n,\delta_n}_{\tau^{\e_n,\delta_n}_j } )}
\Big(
L^{a^{\e_n}_{k_j^{\e_n,\delta}}}_{\tau^{\e_n,\delta_n}_{j+1}}  (X^{\e_n,\delta_n})
-L^{a^{\e_n}_{k_j^{\e_n,\delta_n}}}_{\tau^{\e_n,\delta_n}_{j }} (X^{\e_n,\delta_n})\Big)
\Big|\\
&\leq
\max_{\tau_j^{\e_n,\delta_n}\leq t}
\e_n
\frac{d  (a^{\e_n}_{k_j^{\e_n,\delta_n}})}{\Sigma^{00}(\tau^{\e_n,\delta_n}_j,
a^{\e_n}_{n_j^{\e_n,\delta_n}}, Y^{\e_n,\delta_n}_{\tau^{\e_n,\delta_n}_j } )}
\Big(
L^{a^{\e_n}_{k_j^{\e_n,\delta_n}}}_{\tau^{\e_n,\delta_n}_{j+1}}  (X^{\e_n,\delta_n})
-L^{a^{\e_n}_{k_j^{\e_n,\delta_n}}}_{\tau^{\e_n,\delta_n}_{j }} (X^{\e_n,\delta_n})\Big)
\stackrel{\text{u.c.p.}}{\to} 0.
\ea
Applying \eqref{e:csumL1d} and \eqref{e:h0} finishes the proof.
\end{proof}

The next Lemma is well-known; the proof is omitted.

\begin{lem}
\label{lem:conv_integrals_ordinary}
Let $\{l_n\}_{n\in\mathbb N}$ and $\{h_n\}_{n\in\mathbb N}$ be sequences of real valued c\`adl\`ag functions on $\bR_+$,
and let $l_0$ and $h_0$ be real valued continuous functions on $\bR_+$ such that
    \begin{enumerate}
        \item each function $t\mapsto l_n(t)$, $n\in \mathbb N$, is non-decreasing,
        \item $\lim_{n\to\infty}l_n=l_0$ point-wise,
        \item $\lim_{n\to\infty}h_n=h_0$ locally uniformly.
    \end{enumerate}
Then,
\ba
\lim_{n\to\infty}\int_0^\cdot h_n(s)\, \di l_n(s)=\int_0^\cdot h_0(s)\, \di l_0(s)
\ea
locally uniformly.
\end{lem}

\begin{proof}[Proof of Theorem \ref{thm:conv_CAF}]
Without loss of generality, by Skorokhod's representation theorem,
we may assume that
\ba
(X^{\e_n, \delta_n}, Y^{\e_n, \delta_n},W^{\e_n, \delta_n})\to (X,Y,W),\quad n\to\infty,
\ea
locally uniformly on time intervals with probability 1.
Convergence of the Lebesgue integrals follows from Lebesgue's dominated convergence theorem.
Convergence of the It\^o stochastic integrals follows from Theorem 2.2 in \cite{KurtzP-91}.
It is left to establish convergence of integrals with respect to local times.

Let us introduce increasing continuous processes
\ba
l_0(t)&:=t,\\
l_n(t)&:=\e_n   \sum_k \int_0^t
     \frac{d  (X^{\e_n,\delta_n}_s)}{\Sigma^{00}(s, X^{\e_n,\delta_n}_s, Y^{\e_n,\delta_n}_s )}\, \di L^{a_k^\e}_{s} (X^{\e_n,\delta_n}),
     \quad n\in\mathbb N,
\ea
and continuous functions
\ba
h_0(t)&:= h(t,X_t,Y_t) \frac{\Sigma^{00}(t, X_{t }, Y_{t } )}{d  (X_{t})},\\
h_n(t)&:= h^{\e_n,  \delta_n}(t,X_t^{\e_n,  \delta_n},Y_t^{\e_n,  \delta_n})
\frac{\Sigma^{00}(t,X^{\e_n,\delta_n}_{t }, Y^{\e_n,\delta_n}_{t } )}{d  (X^{\e_n,\delta_n}_{t})},\quad n\in\mathbb N.
\ea
Note that
\ba
\e_n \sum_{k=-\infty}^\infty
\int_0^t h^{\e_n,  \delta_n}(s,a_k^{\e_n},Y_s^{\e_n,  \delta_n})\,\di L^{a_k^{\e_n}}_s (X^{\e_n,\delta_n})
=
\int_0^t h_n(s)\,\di l_n(s).
\ea
The application of Lemmas \ref{e:tot} and \ref{lem:conv_integrals_ordinary} finishes the proof.
\end{proof}

\section{Proof of Theorem \ref{t:A} and Corollary \ref{t:hom}\label{sec:proofs}}

It is sufficient to prove the theorem for any sequence $(\e_n,\delta_n,\lambda_n)_{n\in\mathbb N}$ converging to zero and
satisfying Assumptions $\mathbf{A}_\mathfrak{p}$ and $\mathbf{A}_\mathfrak{q}$.
For brevity, we omit the index $n$ and identify $\e_n=\e$, $\delta_n=\delta$, $\lambda_n=\lambda$.

1. Weak relative compactness. The families of stochastic and Lebesgue integrals
\ba
&\left\{\int_0^\cdot  \sigma_{l}^0(s, X^{\e,\delta}_s,Y^{\e,\delta}_s)\,\di W^l_s\right\}_{\e,\delta\in(0,1]}, \
\left\{\int_0^\cdot b^0(s,X^{\e,\delta}_s,Y_s^{\e,\delta})\,\di s\right\}_{\e,\delta\in(0,1]}, \\
&\left\{ \int_0^\cdot \sigma^{i}_l(s,X_s^{\e,\delta},Y_s^{\e,\delta})\,\di W^l_s\right\}_{\e,\delta\in(0,1]}, \
\left\{\int_0^\cdot b^i(s,X_s^{\e,\delta},Y_s^{\e,\delta})\,\di s\right\}_{\e,\delta\in(0,1]}
\ea
are weakly relatively compact because all the coefficients are bounded by Assumption $\mathbf{A}^{C_b}_\mathrm{coeff}$.

The weak relative compactness of the sequences
\ba
&\left\{\delta \sum_{k=-\infty}^\infty  \int_0^\cdot \beta(s,a_k^\e,Y_s^{\e,\delta})\,\di L^{a_k^\e}_s
(X^{\e,\delta})\right\}_{\e,\delta\in(0,1]} \text{ and }
\left\{\delta \sum_{k=-\infty}^\infty  \int_0^\cdot  \theta^i(s,a_k^\e,Y_s^{\e,\delta})\,\di L^{a_k^\e}_s
(X^{\e,\delta}) \right\}_{\e,\delta\in(0,1]}
\ea
will follow from the next Lemma \ref{cor:weak_comp_localT} and Assumption
that $\mathfrak{p}\in[0,\infty)$.

If additionally, $\mathfrak{q}\in[0,\infty)$, then the family
\ba
\left\{\lambda\sum_{k=-\infty}^\infty
\int_0^t  \gamma(s,a_k^\e,Y_s^{\e,\delta})\,\di L^{a_k^\e}_s (X^{\e,\delta}) \right\}_{\e,\delta,\lambda\in(0,1]}
\ea
is weakly relatively compact, too.

Therefore, the family $\{X^{\e,\delta},Y^{\e,\delta}\}$ is weakly relatively compact, and if
$\mathfrak{q}\in[0,\infty)$, then the
family $\{X^{\e,\delta},Y^{\e,\delta},A^{\e,\delta,\lambda}\}$ is weakly relatively compact, too.

\begin{lem}
\label{cor:weak_comp_localT}
Let $\{\e_n\}_{n\in\mathbb N},\{\delta_n\}_{n\in\mathbb N}\subseteq(0,1]$ be sequences converging to zero.
Let $\{f_\cdot^{\e_n,\delta_n}\}_{n\in\mathbb{N}}$ be a family of measurable
stochastic processes bounded by the same constant $C\in(0,\infty)$.

Then the family
\ba
\label{e:fdL}
\Big\{\e_n \sum_{k=-\infty}^\infty \int_0^\cdot  f_s^{\e_n,\delta_n} \di L^{a_k^{\e_n, \delta_n}}_s
(X^{\e_n, \delta_n})\Big\}_{n\in\mathbb N}
\ea
is weakly relatively compact in $C(\bR_+,\bR)$.
\end{lem}
\begin{proof}
For brevity, we omit the index $n$ and identify $\e_n=\e$, $\delta_n=\delta$.
Let $T\in[0,\infty)$.
To estimate the modulus of continuity of the process from \eqref{e:fdL} we fix two arbitrary time instants $t_1$, $t_2$ such that
$0\leq t_1\leq t_2\leq T$. Then,
\ba
\e \Big| &\sum_{k=-\infty}^\infty \int_{t_1}^{t_2} f_s^{\e,\delta} \di L^{a_k^{\e, \delta}}_s (X^{\e, \delta})\Big|
\leq
 C \sup_{s,x,y}\frac{\Sigma^{00}(s,x,y)}{d(x)}\cdot
\e
\sum_{k=-\infty}^\infty\int_{t_1}^{t_2} \frac{d(X^{\e,\delta}_s)}{\Sigma^{00}(s,X^{\e,\delta}_s, Y^{\e,\delta}_s)}\,
\di L^{a_k^\e}_s (X^{\e,\delta}).
\ea
By Lemma \ref{e:tot}, we see that the r.h.s.\ of the latter inequality
converges in probability
to
\ba
C \sup_{s,x,y}\frac{\Sigma^{00}(s,x,y)}{d(x)} (t_2-t_1)
\ea
uniformly over $t_1,t_2\in[0,T]$, so that the weak relative compactness is established.
\end{proof}

2. Identification of the limit.

It follows from Theorem \ref{thm:conv_CAF} that any limit point of the family $\{X^{\e,\delta},Y^{\e,\delta},W^{\e,\delta}\}$
satisfies the SDE \eqref{e:XYlim}. Since this SDE has a unique (strong) solution by Assumption $\mathbf{A}^\mathrm{Lip}_\mathrm{coeff}$ ,
the convergence \eqref{eq:conv_XY} follows.

Convergence \eqref{eq:conv_XYA} follows analogously with the help of Theorem \ref{thm:conv_CAF}.

The proof of \eqref{eq:limit_hom1} follows from the continuous mapping theorem and
Theorem 13.2.1 in \cite{whitt02}.

Equation \eqref{e:XYhatlim} follows from Section (b) \emph{Time change} on p.\ 225 in
\cite{IW89} because $A^{-1}$ is a random time change in an SDE.

Finally, if $\mathfrak{q}=+\infty$, the limit \eqref{eq:limit_hom2} follows because $(A^{\e,\delta,\lambda})^{-1}$ converges to $0$.

\section{Proof of Theorem \ref{thm:limitODE}}

Recall that
$(\widetilde X^{\e,\delta}_t,\widetilde Y^{\e,\delta}_t)_{t\in[0,\infty)}
:=(X^{\e,\delta}_{t \e/\delta},Y^{\e,\delta}_{t \e/\delta})_{t\in[0,\infty)}$.
It can be seen that
$L^{a_k^\e}_{t\e/\delta} (X^{\e,\delta})=L^{a_k^\e}_{t } (\widetilde X^{\e,\delta})$, see \eqref{e:localtime_tilde_X}.

Similarly to Lemma \ref{l:31}, the process $(\widetilde X^{\e,\delta}_t,\widetilde Y^{\e,\delta}_t)_{t\in[0,\infty)}$ satisfies the SDE
\ba
\label{e:hat_XY}
\widetilde X^{\e,\delta}_t
&=x+\sum_{l=1}^m
\sqrt{\frac{\e}{\delta}}\int_0^t \sigma_{l}^0\Big( \frac{\e}{\delta} s,   \widetilde X^{\e,\delta}_{  s},  \widetilde Y^{\e,\delta}_{s}\Big)\,\di \widetilde W^l_s
+ \frac{\e}{\delta} \int_0^t b^0\Big(\frac{\e}{\delta} s,  \widetilde X^{\e,\delta}_{  s},  \widetilde Y^{\e,\delta}_{s}\Big)\,\di s
+
{\delta }  \sum_k \int_0^{t } \beta\Big(s\frac{\e}{\delta},a_k^\e,\widetilde Y^{\e,\delta}_{s }\Big)\, \di L^{a_k^\e}_{s} (\widetilde X^{\e,\delta})
,\\
\widetilde Y^{i,{\e,\delta}}_t&= {y^i} +\sum_{l=1}^m \sqrt{\frac{\e}{\delta}}\int_0^t
\sigma_{l}^i\Big(\frac{\e}{\delta} s,  \widetilde X^{\e,\delta}_{  s},  \widetilde Y^{\e,\delta}_{s}\Big)\,\di
\widetilde W^l_s + \frac{\e}{\delta}\int_0^t b^i\Big(\frac{\e}{\delta} s,   \widetilde X^{\e,\delta}_{  s},
\widetilde Y^{\e,\delta}_{s}\Big)\,\di s +
{\delta }  \sum_k \int_0^{t } \theta^i\Big(s\frac{\e}{\delta}, a_k^\e,
\widetilde Y^{\e,\delta}_{s }\Big)\, \di L^{a_k^\e}_{s} (\widetilde X^{\e,\delta}),\\
&\quad i=1,\dots,d.
\ea
Note that as $\e/\delta\to 0$, and all the coefficients are bounded, the stochastic and Lebesgue integrals w.r.t.\
$\di \widetilde W$ and $\di s$ converge to zero in probability as
$\e,\delta\to 0$ uniformly over compact time intervals.

Let us investigate the limit behavior of local times.
The following statement can be proved similarly to Lemmas \ref{l:55} and
\ref{cor:conv_CAF_to_t}.
\begin{lem}
\label{lem:LLN_crossings_general}
Assume that conditions of Theorem \ref{thm:limitODE} hold. Then, we have the following ucp-convergence:
\ba\label{eq:11563}
\delta \sum_k \int_0^{t\e/\delta}   
\frac{d  (X^{\e,\delta}_{s})}{\Sigma^{00}(s, X^{\e,\delta}_{s }, Y^{\e,\delta}_{s } )}\, \di L^{a_k^\e}_{s} (X^{\e,\delta})
=
\delta \sum_k \int_0^{t}
\frac{d  (\widetilde X^{\e,\delta}_{s})}{\Sigma^{00}(s\frac{\e}{\delta},
\widetilde X^{\e,\delta}_{s},\widetilde  Y^{\e,\delta}_{s} )}
\, \di L^{a_k^\e}_{s} (\widetilde X^{\e,\delta})
                \to t,\quad  \e,\delta\to0.
\ea
\end{lem}
\begin{proof}
The proof of \eqref{eq:11563} follows the lines of Section
\ref{s:many} and, in particular, of Lemmas \ref{l:55},
\ref{cor:conv_CAF_to_t}.
We have the following ucp-convergence:
\ba
&\sum_{ \frac{\delta}{\e} \tau^{\e,\delta}_{j }  \leq t}
\frac{\delta}{\e} (\tau^{\e,\delta}_{j+1 }-\tau^{\e,\delta}_{j })
\to t,\quad  \e,\delta\to 0,\\
&\sum_{ \frac{\delta}{\e} \tau^{\e,\delta}_{j }  \leq t}
\frac{\delta}{\e} \E\Big[\tau^{\e,\delta}_{j+1 }-\tau^{\e,\delta}_{j } | \rF_{\tau^{\e,\delta}_{j }}\Big]
\to t,\quad \e,\delta\to 0,\\
& \frac{\delta}{\e} \sum_{ \frac{\delta}{\e}\tau^{\e,\delta}_{j }  \leq t}
\frac{\e^2 d^2 (  X^{\e,\delta}_{ \tau^{\e,\delta}_{j }})}{\Sigma^{00}(\tau^{\e,\delta}_j,
X^{\e,\delta}_{ \tau^{\e,\delta}_j},   Y^{\e,\delta}_{  \tau^{\e,\delta}_j} )}
\to t,\quad  \e,\delta\to 0,\\
&\frac{\delta}{\e} \sum_{\frac{\delta}{\e} \tau^{\e,\delta}_{j}\leq t}
\frac{\e d  (X^{\e,\delta}_{\tau^{\e,\delta}_{j }})}{\Sigma^{00}(\tau^{\e,\delta}_j,
X^{\e,\delta}_{\tau^{\e,\delta}_j}, Y^{\e,\delta}_{\tau^{\e,\delta}_j } )}
\E \Big[  L^{{X^{\e,\delta}_{\tau^{\e,\delta}_j}} }_{\tau^{\e,\delta}_{j+1}}  (X^{\e,\delta})
-L^{X^{\e,\delta}_{\tau^{\e,\delta}_j}}_{\tau^{\e,\delta}_j} (X^{\e,\delta})\Big|  \rF_{\tau^{\e,\delta}_j} \Big]
                \to t,\quad  \e,\delta\to 0,
\ea
and, finally,
\ba
\label{eq:conv_cond_to_t2}
\frac{\delta}{\e} \sum_{\frac{\delta}{\e}\tau^{\e,\delta}_{j}\leq t}
\frac{\e d  (X^{\e,\delta}_{\tau^{\e,\delta}_{j }})}
{\Sigma^{00}(\tau^{\e,\delta}_{j } , X^{\e,\delta}_{\tau^{\e,\delta}_{j } }, Y^{\e,\delta}_{\tau^{\e,\delta}_{j } } )}
\Big(L^{{X^{\e,\delta}_{\tau^{\e,\delta}_j}} }_{\tau^{\e,\delta}_{j+1}}  (X^{\e,\delta})
-L^{X^{\e,\delta}_{\tau^{\e,\delta}_j}}_{\tau^{\e,\delta}_j} (X^{\e,\delta})\Big)
                \to t,\quad  \e,\delta\to 0.
\ea
Therefore, similarly to Lemma \ref{e:tot} we get
\eqref{eq:11563}.
\end{proof}

As in Lemma \ref{cor:weak_comp_localT}, convergence \eqref{eq:11563} implies weak relative compactness of the family
\ba
\Big\{{\delta }
\sum_k \int_0^\cdot \beta\Big(\frac{\e}{\delta}s,a_k^\e,\widetilde Y^{\e,\delta}_{s }\Big)\, \di L^{a_k^\e}_{s} (\widetilde X^{\e,\delta}),
\delta  \sum_k \int_0^\cdot \theta^i\Big(\frac{\e}{\delta}s, a_k^\e,
\widetilde Y^{\e,\delta}_{s }\Big)\, \di L^{a_k^\e}_{s} (\widetilde X^{\e,\delta}),\ i=1,\dots,d\Big\}_{\e,\delta\in(0,1]}
\ea
and, consequently, of the family $\{ \widetilde X^{\e,\delta}, \widetilde Y^{\e,\delta}\}_{\e,\delta\in(0,1]}$.

Similarly to the proof of Theorem \ref{thm:conv_CAF}, convergence \eqref{eq:11563} and Lemma \ref{lem:conv_integrals_ordinary}
imply that any limit point of $( \widetilde X^{\e,\delta}, \widetilde Y^{\e,\delta})$ satisfies the ODE
\eqref{e:XY_1membrane-degenODE}.
Since this ODE has a unique solution, we obtain convergence \eqref{eq:limit_isODE}.
If, moreover, $\mathfrak{r} =[0,\infty)$, then we have \eqref{eq:conv_XYA_deg} and, therefore,
\eqref{eq:conv_XYA_comp_d}. The case $\mathfrak{p}=\mathfrak{r} =\infty$ is trivial.

\section*{Acknowledgments}
O.A.\ acknowledges funding from the DFG project AR 1717/2-1 (548113512).
A.P.\  thanks the Swiss National Science
Foundation for its support (grants IZRIZ0\_226875, 200020\_214819, 200020\_200400, and 200020\_192129).
The text of this paper was partially reviewed for spelling and grammar using
ChatGPT.

%
%
%
%
%
%
%
%
%

%

\end{document}